%% file: blowup-arxiv.tex
\documentclass{article}

\usepackage[T1]{fontenc}
\usepackage{graphicx}
\usepackage{amsthm}
\usepackage{amsmath}
\usepackage{amssymb}
\usepackage{thmtools,thm-restate}
\usepackage{stmaryrd}
\usepackage{xcolor}
\usepackage{bbold}
\usepackage{dsfont}
\usepackage{enumerate}
\usepackage{wrapfig}
\usepackage{array}
\usepackage{tikz}
\usetikzlibrary{backgrounds}
\usetikzlibrary{cd}
\usetikzlibrary{patterns}
\usetikzlibrary{calc}
\usetikzlibrary{intersections}
\usetikzlibrary{decorations.pathmorphing}
\usetikzlibrary{decorations.pathreplacing}
\usetikzlibrary{calligraphy}
\usetikzlibrary{decorations.markings}
\usetikzlibrary{arrows,arrows.meta}
\usetikzlibrary{hobby}
\usepackage[capitalize]{cleveref}
\usepackage{macros}
\usepackage{fontawesome5}
\usepackage{authblk}

\theoremstyle{plain}
\newtheorem{theorem}{Theorem}[section]
\newtheorem{lemma}[theorem]{Lemma}
\newtheorem{proposition}[theorem]{Proposition}
\newtheorem{corollary}[theorem]{Corollary}
\newtheorem{definition}[theorem]{Definition} 

\theoremstyle{definition}

\newtheorem{remark}[theorem]{Remark}
\newtheorem{claim}[theorem]{Claim}

\title{Non-Hausdorff manifolds over locally ordered spaces via sheaf theory}
\author[1]{Yorgo Chamoun\thanks{Email : yorgo.chamoun@polytechnique.edu}}
\author[1]{Emmanuel Haucourt}
\affil[1]{École Polytechnique, Laboratoire d’informatique (LIX), Bâtiment Alan Turing, 1 rue Honoré d’Estienne d’Orves, 91120 Palaiseau Cedex, France}
\date{}

\begin{document}

\maketitle

\begin{abstract}
  \emph{Locally ordered spaces} can be used as topological models of concurrent programs: the local order models the irreversibility of time during execution. Under certain conditions, one can even work with \emph{locally ordered manifolds}. In this paper, we build the universal \emph{euclidean} local order over every locally ordered space; in categorical terms, the subcategory of euclidean local orders is \emph{coreflective} in the category of locally ordered spaces. Our construction is based on a well-known correspondance between \emph{sheaves} and \emph{étale bundles}. This is a far reaching generalization of a result about \emph{realizations} of graph products. We particularize the construction to locally ordered realization of \emph{precubical sets}, and show that it admits a purely combinatorial description. With the same proof techniques, we show that, unlike for the topological realization, there is a unique (up to symmetry) precubical set whose locally ordered realization is isomorphic to $\R^n$.\\

\end{abstract}

\input{introduction}

\section{Preliminaries}\label{sec_prelim}
\subsection{Sheaves and presheaves}
The paper heavily relies on the machinery of \emph{sheaves} (on topological spaces) and \emph{presheaves} (on an arbitrary category) \cite{maclane2012sheaves} (\S II.1 and \S I.1 (vii) respectively).

A \emph{presheaf} on a category $\ccat$ is a functor $P:\ccat^{op}\xrightarrow{}\Set$, i.e.~a collection of sets $P(c)$, for $c$ object of $\ccat$, and a collection of maps $P(f):P(c)\xrightarrow{}P(d)$ for $f:d\xrightarrow{}c$ arrow of $\ccat$, satisfying $P(f\circ g)=P(g)\circ P(f)$ and $P(\id_c)=\id_{P(c)}$ for all morphisms $f,\, g$ and objects $c$ (where $\id_c$ denotes the identity morphism at $c$). 

A \emph{sheaf} on a topological space $X$ is a presheaf on the posetal category $O(X)$ of open sets of $X$ satisfying an amalgamation property with respect to open covers. Namely, a presheaf $F$ on $O(X)$ is a sheaf if and only if for every open set $U$ and every open cover $(U_i)_{i\in I}$ of $U$, the following diagram is an equalizer:
\begin{center}
    \begin{tikzcd}
        F(U)\arrow[r] & \prod_{i\in I}F(U_i)\arrow[r,shift left=0.5ex]\arrow[r,shift right=0.5ex] & \prod_{i,j\in I}F(U_i\cap U_j)
    \end{tikzcd}
\end{center}
where the upper morphism is given by the $\prod_{i\in I}F(U_i)\xrightarrow{}F(U_i)\xrightarrow{}F(U_i\cap U_j)$ and the lower morphism is given by the $\prod_{i\in I}F(U_i)\xrightarrow{}F(U_j)\xrightarrow{}F(U_i\cap U_j)$. The standard example is the sheaf which associates to an open set $U$ the set of continuous functions $U\xrightarrow{}\R$; the above diagram  being an equalizer just says that a continuous function $f$ on $U$ is equivalent to a family of continuous functions $f_i:U_i\xrightarrow{}\R$ for $(U_i)_{i\in I}$ an open cover of $U$ satisfying $(f_i)_{|U_i\cap U_j}=(f_j)_{|U_i\cap U_j}$ for all $i,\, j\in I$. Thus, sheaves can be very useful in situations where we want to study the interplay between local and global phenomena.

\subsection{Graphs and precubical sets}\label{sec:precubical_sets}
Let $\mathcal{G}$ be the category with two objects and two parallel morphisms $0 \rightrightarrows 1$. A \emph{graph} is a presheaf on $\mathcal{G}$. Given a graph $P$, $P(0)$ is the \emph{set of vertices}, $P(1)$ the \emph{set of arrows}, and the functions $P(1) \rightrightarrows P(0)$ are the \emph{source} and the \emph{target}, noted respectively $s$ and $t$.

Precubical sets can be naturally seen as the generalization of graphs in higher dimension, in the following sense. Let $\square$ be the free monoidal category on $\mathcal{G}$ with unit $0$. In particular, the objects are the natural numbers. A \emph{precubical set} can then be defined as a presheaf on $\square$. More explicitly, a precubical set $P$ is a collection of sets $P(n)$ (or $P_n$ for convenience) which are called the \emph{set of $n$-cubes} of $P$ (for $n$ a natural number), and a collection of maps $P(n)\xrightarrow{}P(k)$, $k\leq n$, which are called the \emph{$k$-face maps}. More precisely, we note $\delta_{n,i}^\epsilon:P(n)\xrightarrow{}P(n-1)$, with $\epsilon\in\{-,+\}$ (or $\{0,1\}$), $0\leq i<n$ the $(n-1)$-face maps corresponding to every dimension, whose interpretation is the following: $\delta_{n,i}^\epsilon$ associates to an $n$-cube its back ($\epsilon=-$) or front ($\epsilon=+$) face in direction $i$. Any $k$-face is obtained by successively applying maps of the form $\delta_{k,i}^\epsilon$, and these are subject to the cocubical relations 
\[\delta^\epsilon_{k,j}\delta^\eta_{k+1,i}=\delta^\eta_{k,i}\delta^\epsilon_{k+1,j+1}\]
for $0\leq i\leq j<n$. This completely describes a precubical set $P$. For example, the precubical set freely generated by a single $2$-cube $c$ is given by:
\begin{center}
    \begin{tikzpicture}
    \fill[mygray] (0,0) rectangle (2,2);
      \filldraw (0,0) circle (0.05);
      \filldraw (2,0) circle (0.05);
      \filldraw (0,2) circle (0.05);
      \filldraw (2,2) circle (0.05);
      \draw
      (0,0) edge[thick,"$\delta_{2,1}^-(c)$"'] (2,0)
      (2,0) edge[thick,"$\delta_{2,0}^+(c)$"'] (2,2)
      (0,0) edge[thick,"$\delta_{2,0}^-(c)$"] (0,2)
      (0,2) edge[thick,"$\delta_{2,1}^+(c)$"] (2,2)
      ;
      \draw (1,1) node{$c$};
      \draw (0,0) node[below left]{\tiny $\delta_{1,0}^-\circ\delta_{2,0}^-(c)$};
      \draw (0,2) node[above left]{\tiny $\delta_{1,0}^+\circ\delta_{2,0}^-(c)$};
      \draw (2,0) node[below right]{\tiny $\delta_{1,0}^+\circ\delta_{2,1}^-(c)$};
      \draw (2,2) node[above right]{\tiny $\delta_{1,0}^+\circ\delta_{2,1}^+(c)$};
\end{tikzpicture}
\end{center}

\noindent A face of a cube obtained by successively applying morphisms of the form $\delta^+_{k,i}$ (respectively $\delta^-_{k,i})$ is called an \emph{upper} (respectively \emph{lower}) face. A vertex which is an upper (respectively lower) face is called the \emph{maximum} (respectively the \emph{minimum}) of the cube. The \emph{dimension} of a precubical set $P$ is the maximal $n$ such that $P_n\ne 0$ if it exists, and $\aleph_0$ if it does not exist. A precubical set $P$ is called \emph{$n$-homogeneous} if the set of $k$-cubes of $P$ for $k\in\N$ is exactly the set of $k$-faces of $n$-cubes of $P$. Two cubes in a precubical set $P$ are \emph{adjacent} if one is a face of the other. The \emph{tensor product} of two precubical sets $P$ and $Q$ is defined by $(P\otimes Q)_n=\bigcup_{i+j=n}P_i\times Q_j$ with the expected face maps. See for example~\cite[\S 3.4.1]{fajstrup2016directed}. 

Note that the category of precubical sets can be seen as a subcategory of the category of \emph{symmetric} precubical sets, which is the category of presheaves over the free \emph{symmetric} monoidal category on $\mathcal{G}$ with unit $0$, and which we note $\square_\mathsf{sym}$ (see~\cite[\S 6]{grandis2003cubical} for details). More concretely, a symmetric precubical set $P$ is a precubical set for which every $P(n)$ is equipped with an action of the symmetric group $\mathfrak{S}_n$ (compatible with the face maps in an obvious way). Every precubical set can be seen as a symmetric precubical set, by freely adding the images of these actions (this is the left adjoint to the precomposition by the inclusion $\square\xrightarrow{}\square_\mathsf{sym}$). In what follows, by isomorphic precubical sets (noted $\cong$), we mean precubical sets which are isomorphic as symmetric precubical sets. This just means that we allow ourselves to permute the dimensions of a cube: intuitively, an $n$-cube of a symmetric precubical set $P$ is represented by the \emph{orbit} of some element of $P(n)$. This will be crucial, since we will be looking for combinatorial analogs of embeddings of ordered spaces. For example, the map $(x,y)\mapsto(y,x)$ is indeed an automorphism of $\R^2$ equipped with the product topology and order.

\subsection{Locally ordered spaces}\label{sec:locally_ordered_spaces}
An \emph{ordered space} is a topological space equipped with an order relation. A morphism of ordered spaces is a continuous order-preserving map. An isomorphism of ordered spaces is called a \emph{dihomeomorphism}. However, this notion is too strict in general. For example, we would like to define a directed path in $X$ to be a morphism $[0,1]\xrightarrow{}X$, where $[0,1]$ has the natural topology and order, but then no directed \emph{loop} exists in $X$, which is a problem as soon as we try to define a satisfying directed version of the circle. This motivates the following definition: a \emph{locally ordered space}, or \emph{local order}, is a topological space $X$ which has an open base of ordered sets. More formally, let $X$ be a topological space. An \emph{ordered base} $\mathcal{B}$ of $X$ is a collection of (partially) ordered sets satisfying the following conditions: 
\begin{itemize}
    \item the underlying sets of the elements of $\mathcal{B}$ form a base for the topology of $X$,
    \item for every $B,\,B'\in\mathcal{B}$ and $x\in B\cap B'$, there is $B''\in\mathcal{B}$ such that $x\in B''\subseteq B\cap B'$ and $B''$ is a subposet of $B$ and $B'$.
\end{itemize}
\noindent A locally ordered space is defined to be a topological space equipped with an ordered base $\mathcal{B}$. The elements of $\mathcal{B}$ are then called \emph{ordered open sets}. We can now define the directed circle to be the local order whose underlying topological space is the circle $\mathbb{S}^1$, and whose ordered open sets are the proper open sets of $\mathbb{S}^1$ with the order induced by, say, the anti-clockwise direction. These are clearly order-compatible. See the picture below, where we have represented two elements of the ordered base in grey, with the order induced by the arrow.
\begin{center}
    \begin{tikzpicture}
        \draw (0,0) circle (1cm);
        \draw[decoration={markings,mark=at position 0.5 with \arrow{stealth}}, postaction={decorate}] (0,0) ++(130:1cm) arc (130:330:1cm);
        \draw[line width=1.5mm,opacity=0.3,line cap=round] (0,0) ++(30:1cm) arc (30:130:1cm);
        \draw [line width=1.5mm,opacity=0.3,line cap=round] (0,0) ++(80:1cm) arc (80:180:1cm);
    \end{tikzpicture}
\end{center}

A morphism of local orders (respectively a \emph{local embedding}) $f:X\xrightarrow{}Y$ is a continuous map which is locally order preserving (respectively locally an \emph{embedding}), i.e.~for every $x\in X$ and ordered neighborhood $V$ of $f(x)$, there exists an ordered neighborhood $U$ of $x$ such that $f(U)\subseteq V$ and $f_{|U}$ is order preserving (respectively order preserving and admitting a continuous order preserving inverse on its image). A locally ordered space is \emph{euclidean} (of dimension $n$) if every point has an ordered neighborhood which is dihomeomorphic to $\R^n$ (with the product order). Note that euclidean local orders can be thought of as 'locally ordered (topological) manifolds'. A \emph{directed path} is a morphism of locally ordered spaces whose domain is euclidean of dimension 1 and connected. We note $\mathbb{L}$ the category of local orders and locally ordered maps, $\mathbb{L}_e$ the category of locally ordered spaces and local embeddings, and $\mathcal{E}_n$ the category of euclidean local orders of dimension $n$ and local embeddings. Equivalent definitions of locally ordered spaces are listed in~\cite{coursolle2024non}.

\subsection{Locally ordered realization of precubical sets}\label{sec:lorps}
In algebraic topology, we are used to having a realization functor from combinatorial to continuous models. For example, every precubical set $P$ admits a geometric realization $|P|$, which is simply given by left Kan extension. An explicit description can be found in~\cite[Definition 6.5]{fajstrup2006algebraic}; we mention that the underlying set of $|P|$ is $\bigcup_{n\in\N}P_n\times]0,1[^n$, and that $|P|$ is sequential since it is a quotient of sequential spaces. For $c\in P_n$, $|c|$ will mean the realization of the precubical subset of $P$ \emph{spanned} by $c$, i.e.~$c$ and all its faces in all dimensions. Note that $|c|$ is closed in $|P|$. Also note that $\{c\}\times ]0,1[^n$ is open and included in $\mathring{|c|}$ the interior of $|c|$, and they are in fact equal when all the faces of $c$ are distinct. For an open subset $U$ of $|P|$ and a cube $c$, we say that $c$ is \emph{a cube of $U$}, or $U$ \emph{intersects} $c$, if $U\cap \{c\}\times]0,1[^{\dim(c)}\ne\varnothing.$

In the directed case, the situation is more complicated because the category of locally ordered spaces, unlike the category of topological spaces, is not cocomplete~\cite{coursolle2024non}. 
Given a precubical set $P$, we call \emph{locally ordered realization} of \(P\) any locally ordered space $|P|_{lo}$ (with ordered base $\mathcal{B}$) satisfying the following conditions:
\begin{enumerate}
    \item [(i)] The underlying topological space of $|P|_{lo}$ is the geometric realization of $P$ (in particular the underlying set is the same).
    \item [(ii)] For every $B\in\mathcal{B}$, the preorder given by $x\leq y$ if and only if there is a path from $x$ to $y$ that is increasing on each cube
  that it traverses (which we call cubewise increasing) is a partial order and coincides with the order on $B$.
    \item [(iii)] For every $B\in\mathcal{B}$, $B$ is \emph{order convex} in the sense that for every $x,y\in B$, if $x\in\{c\}\times]0,1[^{\dim(c)}$ and $y\in\{c'\}\times]0,1[^{\dim(c')}$ with $c$ a lower (respectively upper) face of $c'$ and such that the segment $]x,y[$ (respectively $]y,x[$) is directed in $\{c'\}\times]0,1[^{\dim(c')}$, then $x\leq y$ (respectively $y\leq x$) in $B$.
\end{enumerate}

\noindent Note that there is at most one such realization up to isomorphism, so we can soundly talk about \emph{the} locally ordered realization of a precubical set. Such a realization is known to exist for the (large) class of \emph{non-self-linked} precubical sets~\cite[\S6.5]{fajstrup2006algebraic}; it actually exists in even bigger generality, as will be shown in a forthcoming paper of the second author. We will also suppose that the ordered base $\mathcal{B}$ satisfies the following property, which will be of technical use (see Lemma~\ref{cond4} and the discussion after it):

\begin{enumerate}
\setcounter{enumi}{3}
    \item [(iv)] For every vertex $v$ of $P$, an ordered neighborhood $B$ of $v$ only intersects cubes adjacent to $v$, and for every cube $c$ with minimum (respectively maximum) $v$, for every point $x\in\{c\}\times]0,1[^{\dim(c)}$, every point $y\geq x$ (respectively $y\leq x$) in $B$ belongs to $\{c'\}\times]0,1[^{\dim(c')}$ for some cube $c'$ having $c$ as lower (respectively upper) face.
\end{enumerate}

In this paper, we implicitely only consider precubical sets which have a locally ordered realization.

\section{Colimits of local orders}\label{sec_colimit}
\begin{definition}
    Let $X$ be a locally ordered space with ordered base $\mathcal{B}$, and $U\subseteq X$ an open subset of $X$. Then $U$ can naturally be seen as a locally ordered space via the ordered base $\{B\cap U\,|\,B\in\mathcal{B}\}$. We say that $U$ is equipped with the \emph{induced local order}. 
\end{definition}

\noindent As mentioned before, the category $\mathbb{L}$ is not cocomplete. However, some useful colimits do exist. Let $(X_i)_{i\in I}$ be a family of local orders. Suppose that there exists a family $(X_{ij},\phi_{ij})_{i,j\in I}$ such that the underlying set of $X_{ij}$ is an open subset of $X_i$ and $\phi_{ij}$ is an isomorphism of locally ordered spaces $X_{ij}\cong X_{ji}$ (where the local orders are the induced ones), subject to the following conditions:
\begin{enumerate}
    \item $\phi_{ij}^{-1}=\phi_{ji}$
    \item $\phi_{ij}(X_{ij}\cap X_{ik})=X_{ji}\cap X_{jk}$
    \item $\phi_{jk}\circ\phi_{ij}=\phi_{ik}$ on $X_{ij} \cap X_{ik}$
\end{enumerate}
Then we can consider the diagram \(D\) of local orders consisting of the $(X_i)_i$ and the $(X_{ij})_{ij}$ together with the inclusions $\iota_{ij}:X_{ij}\hookrightarrow{}X_i$ and the isomorphisms $(\phi_{ij})_{ij}$. 

\begin{proposition}
    The diagram \(D\) admits a colimit in $\mathbb{L}$.
\end{proposition}
\begin{proof}
    First we form its colimit in the category of topological spaces. Note $X$ this colimit. Using~\cite[(3.2.13)]{ducros2014introduction}, we know that the structure maps $X_i\xrightarrow{}X$ are \emph{open immersions}, in the sense that they induce an isomorphism (of topological spaces) between $X_i$ and an open subset of $X$, and $X$ is the union of (the image of) all the $X_i$, $i\in I$, the intersection of $X_i$ and $X_j$ in $X$ being (the image of) $X_{ij}$ (or equivalently $X_{ji}$). Now we have to exhibit an ordered base on $X$. For every $i\in I$, there is an ordered base $(\bindexp ki)_{k\in K(i)}$ of $X_i$. Consider the collection of (the image of) all these families in $X$. We just need to show that if $x\in \bindexp ki \cap \bindexp{k'}j$ then there is some $\bindexp{k''}l$ containing $x$, included in $\bindexp ki \cap \bindexp{k'}j$ and such that the corresponding orders coincide. This just follows from the fact that $X_{ij}$ and $X_{ji}$ are isomorphic as locally ordered spaces under $X$, and that $(\bindexp ki)_{k\in K(i)}$  (resp. $(\bindexp kj)_{k\in K(j)}$) is an ordered base of $X_i$ (resp. $X_j$). Now clearly this is the `maximal' local order structure on $X$ such that the structure maps are local order morphisms, and in fact we get that they are local order isomorphisms on their image. The universal property of the colimit follows from this and from the universal property of the colimit in the category of topological spaces, which concludes the construction.
\end{proof}

\begin{definition}
    We call the above construction the \emph{gluing} of the family $(X_i)$ along $(X_{ij})_{ij}$.
\end{definition}

\section{Blowing up via sheaves}\label{sec_blowup}
As suggested in the introduction, the blowup of a local order $X$ is a bundle whose fibers over $x\in X$ are the equivalence classes of ways of embedding $\R^n$ into $X$ by sending $0$ to $x$. In fact, there is a construction in sheaf theory which essentially does exactly this. This will lead to a vast generalization of the construction of Theorem~\ref{thm:blowup_graphs}, to all local orders, and we will still have the same nice universal property. 

\subsection{The blowup}
Recall that for a sheaf $F$ on $X$ and a point $x\in X$, the \emph{stalk} of $F$ at $x$ is given by $F_x:=\colim_{V\ni x} F(V)$. Given an open set $U\ni x$, the \emph{germ} $A_x$ of $A\in F(U)$ at $x$ is the equivalence class of $A$ in $F_x$. The corresponding equivalence relation is noted $\sim_x$.

\begin{definition}
    Let $X$ be a local order. We define the \emph{sheaf of $n$-traversals} $\oindexp{X}{n}$ by the setting:
    \[\oindexp{X}{n}(U):=\{A\subseteq U\,|\,A\cong E\text{ with $E$ some $n$-euclidean local order}\}\]
    for an open set $U\subseteq X$, the restriction map $\oindexp{X}{n}(V)\xrightarrow{}\oindexp{X}{n}(U)$ for $U\subseteq V$ being given by intersection with $U$.
\end{definition}

\begin{remark}
\begin{enumerate}
        \item The restriction maps are well-defined because an open subset of a euclidean local order is euclidean.
        \item This is \textit{a priori} only a presheaf (but we will prove that it is a sheaf), and the sections are not $n$-traversals (but we will prove that the germs essentially are).
        \item The empty set $\varnothing$ is trivially a euclidean local order, so we always have an element $\varnothing_x\in\oindexp{X,x}{n}$. This is in fact crucial to make $\oindexp{X}{n}$ into a presheaf.
    \end{enumerate}
\end{remark}

\noindent We can think of $\oindexp{X}{n}(U)$ as the set of equivalence classes of euclidean embeddings in $U$ with respect to the relation 'having the same image'. So for a euclidean embedding $f$, we may write $[f:E\xrightarrow{}U]$ to mean $\Im(f)$. Here embedding just means that $f$ is an isomorphism of locally ordered spaces on its image.

\begin{proposition}
    For every local order $X$, $\oindexp{X}{n}$ is a sheaf.
\end{proposition}
\begin{proof}
    We have to prove that given a family of open sets $(U_i)_{i\in I}$ and a family of subsets $(A_i)_{i\in I}$ such that $A_i\in \oindexp{X}{n}(U_i)$ and $A_i\cap U_j=A_j\cap U_i$ for all $i,j\in I$, there is a unique $A\in \oindexp{X}{n}(U)$ such that $A\cap U_i=A_i$ for all $i$. So consider such family. Uniqueness is obvious since we have local uniqueness on an open cover. So we just need to prove existence. Clearly if $A$ exists then $A=\bigcup_i A_i$, so set $A$ like this. Choose a representative $[f_i:E_i\xrightarrow{}U_i]=A_i$ for all $i$. We want to use \S\ref{sec_colimit} to glue the $E_i$. For every $i,j$, define $E_{ij}:=f_i^{-1}(A_i\cap A_j)$. Notice that since $\Im(f_i)=A_i$, we have
    \[f_i^{-1}(U_j)=f_i^{-1}(U_j\cap A_i)=f_i^{-1}(A_j\cap U_i)=f_i^{-1}(A_j\cap U_i\cap A_i)=E_{ij}\]
    so $E_{ij}$ is an open subset of $E_i$. Let $\phi_{ij}:=f_j^{-1}\circ f_i:E_{ij}\cong E_{ji}$. Taking $X_i:=E_i$ and $X_{ij}:=E_{ij}$, all the conditions of \S\ref{sec_colimit} are trivially verified, and we note $E$ the colimit representing the gluing of the $E_i$ along the $E_{ij}$. By the universal property of the colimit, the family $(f_i)_{i\in I}$ defines a map $f:E\xrightarrow{}A$. $E$ is euclidean and $f$ is an isomorphism whose inverse is given by the morphism induced by the family $(f_i^{-1})_{i\in I}$ (note that $A$ is the gluing of the family $A_i$), because these properties can be checked locally. This finishes the proof.
\end{proof}

\begin{proposition}\label{traversal}
    For a local order $X$ and $A\in \oindexp{X}{n}(U)$ for some $U$, if $x\in X$ is such that $x\in A$, then $A_x$ is an $n$-traversal. Conversely, every $n$-traversal is of this form.
\end{proposition}
\begin{proof}
    This is a direct consequence of the following two facts: every isomorphism of ordered spaces is an isomorphism of locally ordered spaces, and every isomorphism of locally ordered spaces is locally (i.e.~up to restriction) an isomorphism of ordered spaces. 
\end{proof}

\noindent One could argue that, in view of Definition~\ref{def:traversal}, we are only interested in the images of \emph{order} embeddings from an open subset of $\R^n$ to some ordered open set of $X$, but this would not give a sheaf, only a presheaf. In fact, taking the sheafification of this presheaf would give exactly $\oindexp{X}{n}$, which justifies our definition. This is an easy consequence of the fact that they have the same set of germs at every point $x\in X$, and essentially expresses the fact that a euclidean local order is an amalgamation of copies of $\R^n$.

An \emph{étale bundle} over a space $X$ is a morphism of topological spaces $f:Y\xrightarrow{}X$ such that for every $y\in Y$, there is a neighborhood $U$ of $y$ such that $f(U)$ is open and $f_{|U}:U\to f(U)$ is an isomorphism of topological spaces. Recall that there is an equivalence of categories between étale bundles over a topological space $X$ and sheaves on $X$~\cite[\S II.6]{maclane2012sheaves}. This was extended in~\cite[Theorem 5.4]{bubenik2005model} to the locally ordered setting, i.e.~replacing 'topological space' by 'locally ordered space'. In fact, we can reformulate this result by noticing that a structure of local order can be pulled back in a unique way along a local homeomorphism, like for smooth manifold structures. We want to use this operation in order to blow up any local order. The next proposition is just a translation of the mentioned results in our case. We fix a natural number $n$.

\begin{proposition}\label{description_etale}
    Let $X$ be a local order. The étale bundle corresponding to $\oindexp{X}{n}$, noted $\beta_X^+:B(\oindexp{X}{n})\xrightarrow{}X$, can be described as follows: 
    \begin{itemize}
        \item Its underlying set map is given by the projection $\sqcup_{x\in X}\oindexp{X,x}{n}\xrightarrow{}X$;
        \item An ordered base is given by the sets $U^+_A:=\{A_y\,|\,y\in U\}$ for $U$ ordered open set, $A\in \oindexp{X}{n}(U)$; the order relation is the one inherited from $U$.
    \end{itemize}
    Moreover, for every $U$ open and $A\in \oindexp{X}{n}(U)$, the function $U\xrightarrow{}B(\oindexp{X}{n})$, $x\mapsto A_x$ is open and an isomorphism of local orders on its image, and $\beta_X^+$ is a local dihomeomorphism.
\end{proposition}

\noindent Now we have an étale bundle over $X$. However, the blowup of $X$ cannot possibly be an étale bundle, since it is always euclidean (and not $X$). Recall that for every ordered open $U\subseteq X$, the empty set belongs to $\oindexp{X}{n}(U)$. This gives a `ghost version' $\sqcup_{x\in X}\{\varnothing_x\}\subseteq B(\oindexp{X}{n})$ of $X$ in $B(\oindexp{X}{n})$, which `completes' the blowup into an étale bundle. We then just need to remove it. An easy example is shown in Fig.~\ref{fig:H}. The bottom space is the realization of the graph $G$ of Fig.~\ref{fig:G}. The upper space is (a subset of, see below) $B(\oindexp{|G|}{1})$. Every column features some $A\in \oindexp{G}{1}(|G|)$ which corresponds to the highlighted part of $|G|$ (in the first column, no part is highlighted, so $A:=\varnothing$). The highlighted part of $B(\oindexp{|G|}{1})$ then corresponds to the open set $|G|_A^+$ (so in the first column, the higlighted part is exactly $\sqcup_{x\in X}\{\varnothing_x\}$). Note that among the five points in the preimage of the branching point of $G$, exactly one is in $|G|_A^+$ for each $A$. It is then easy to see that $|G|_A^+\cong|G|$, i.e.~that $\beta_{|G|}^+$ is étale. Now we get the blowup by removing the ghost version, and we indeed get the space described in the introduction, which is identical to $|G|$ over the edges of $G$, and with four points over the branching point corresponding to the four $1$-traversals at this point. Moreover, the induced topology is exactly the topology of the blowup as described in~\cite{haucourt2024non}.

However, there is a subtelty. In light of Proposition~\ref{traversal}, we have to remove all the $A_x$ with $x\not\in A$. These are not necessarily of the form $\varnothing_x$, for example one can consider the germ of $]0,+\infty[$ at $0$ in $\R$. This subtlety will vanish when we will consider the combinatorial case, since all the `combinatorial embeddings' will be 'maximal' in an appropriate way. To give a first intuition, if we take again the example of $\R$, we see that the germ of $]0,+\infty[$ at a negative point $x$ is $\varnothing_x$, and the germ at a positive point is a traversal, so the germ at $0$ serves as a `transition' between the traversals and the germs of the empty set. But when the embedding is `maximal', for example the embedding of $\R$ in $\R$, no transition is needed. This subtelty is omitted in Fig.~\ref{fig:H}.
\begin{figure}[ht!]
\begin{center}
\begin{tikzpicture}[scale=0.8]
\draw (1.7,5.2) node[right]{\small Neighborhoods of the inverse images};
\draw (1.7,4.63) node[right]{\small of the intersection point};
\draw[densely dotted] (-1,1/3) -- (1,-1/3) ;
\draw[densely dotted] (-1,-1/3) -- (1,1/3) ;
\draw[->,>=stealth] (0,1.3) -- (0,0.7) ;
\begin{scope}[yshift=+35mm]
\draw[ghostgray] (0,0.72) node{{\fontsize{40}{30}\selectfont \faGhost}};
\draw (-1,2/3) -- (1,0) ;
\draw (-1,0) -- (1,2/3) ;
\draw[densely dotted] (-1,-1) -- (-1/6,-1) ;
\draw[densely dotted] (1/6,-1) -- (1,-1) ;
\draw[densely dotted] (-1,-5/3) -- (-1/6,-5/3) ;
\draw[densely dotted] (1/6,-5/3) -- (1,-5/3) ;
\filldraw[fill=white,opacity=0.2] (0,-1) circle (0.02);
\filldraw[fill=white,opacity=0.2] (0,-5/3) circle (0.02);
\filldraw[fill=white,opacity=0.2] (0,-11/9) circle (0.02);
\filldraw[fill=white,opacity=0.2] (0,-13/9) circle (0.02);
\end{scope}
\begin{scope}[xshift=30mm]
\draw (-1,1/3) -- (0,0) -- (1,1/3) ;
\draw[densely dotted] (-1,-1/3) -- (0,0) -- (1,-1/3) ;
\draw[->,>=stealth] (0,1.3) -- (0,0.7) ;
\begin{scope}[yshift=+35mm]
\draw[densely dotted] (-1,2/3) -- (0,1/3) -- (1,2/3) ;
\draw (-1,0) -- (0,1/3) -- (1,0) ;
\filldraw[fill=white] (0,1/3) circle (0.035) ;
\draw (-1,-1) -- (-1/6,-1) ;
\draw (1/6,-1) -- (1,-1) ;
\draw[densely dotted] (-1,-5/3) -- (-1/6,-5/3) ;
\draw[densely dotted] (1/6,-5/3) -- (1,-5/3) ;
\filldraw[fill=white,opacity=0.2] (0,-1) circle (0.02);
\filldraw[fill=white,opacity=0.2] (0,-5/3) circle (0.02);
\filldraw[fill=white,opacity=0.2] (0,-11/9) circle (0.02);
\filldraw[fill=white,opacity=0.2] (0,-13/9) circle (0.02);
\fill (0,-1) circle (0.02);
\end{scope}
\end{scope}
\begin{scope}[xshift=60mm]
\draw (-1,1/3) -- (1,-1/3) ;
\draw[densely dotted] (-1,-1/3) -- (1,1/3) ;
\draw[->,>=stealth] (0,1.3) -- (0,0.7) ;
\begin{scope}[yshift=+35mm]
\draw[densely dotted] (-1,2/3) -- (1,0) ;
\draw (-1,0) -- (1,2/3) ;
\draw (-1,-1) -- (-1/6,-1) ;
\filldraw[fill=white] (0,1/3) circle (0.035) ;
\draw[densely dotted] (1/6,-1) -- (1,-1) ;
\draw[densely dotted] (-1,-5/3) -- (-1/6,-5/3) ;
\draw (1/6,-5/3) -- (1,-5/3) ;
\fill (0,-11/9) circle (0.02);
\filldraw[fill=white,opacity=0.2] (0,-1) circle (0.02);
\filldraw[fill=white,opacity=0.2] (0,-5/3) circle (0.02);
\filldraw[fill=white,opacity=0.2] (0,-11/9) circle (0.02);
\filldraw[fill=white,opacity=0.2] (0,-13/9) circle (0.02);
\end{scope}
\end{scope}
\begin{scope}[xshift=90mm]
\draw[densely dotted] (-1,1/3) --  (1,-1/3) ;
\draw (-1,-1/3) --  (1,1/3) ;
\draw[->,>=stealth] (0,1.3) -- (0,0.7) ;
\begin{scope}[yshift=+35mm]
\draw (-1,2/3) -- (1,0) ; 
\draw[densely dotted] (-1,0) -- (1,2/3) ;
\draw[densely dotted] (-1,-1) -- (-1/6,-1) ;
\draw (1/6,-1) -- (1,-1) ;
\draw (-1,-5/3) -- (-1/6,-5/3) ;
\filldraw[fill=white] (0,1/3) circle (0.035) ;
\draw[densely dotted] (1/6,-5/3) -- (1,-5/3) ;
\filldraw[fill=white,opacity=0.2] (0,-1) circle (0.02);
\filldraw[fill=white,opacity=0.2] (0,-5/3) circle (0.02);
\filldraw[fill=white,opacity=0.2] (0,-11/9) circle (0.02);
\filldraw[fill=white,opacity=0.2] (0,-13/9) circle (0.02);
\fill (0,-13/9) circle (0.02);
\end{scope}
\end{scope}
\begin{scope}[xshift=120mm]
\draw[densely dotted] (-1,1/3) -- (0,0) -- (1,1/3) ;
\draw (-1,-1/3) -- (0,0) -- (1,-1/3) ;
\draw[->,>=stealth] (0,1.3) -- (0,0.7) ;
\begin{scope}[yshift=+35mm]
\draw (-1,2/3) -- (0,1/3) -- (1,2/3) ;
\draw[densely dotted] (-1,0) -- (0,1/3) -- (1,0) ;
\draw[densely dotted] (-1,-1) -- (-1/6,-1) ;
\draw[densely dotted] (1/6,-1) -- (1,-1) ;
\filldraw[fill=white] (0,1/3) circle (0.035) ;
\draw (-1,-5/3) -- (-1/6,-5/3) ;
\draw (1/6,-5/3) -- (1,-5/3) ;
\filldraw[fill=white,opacity=0.2] (0,-1) circle (0.02);
\filldraw[fill=white,opacity=0.2] (0,-5/3) circle (0.02);
\filldraw[fill=white,opacity=0.2] (0,-11/9) circle (0.02);
\filldraw[fill=white,opacity=0.2] (0,-13/9) circle (0.02);
\fill (0,-5/3) circle (0.02);
\end{scope}
\end{scope}
\end{tikzpicture}

\caption{The étale bundle in the simple case of the graph $G$ of Fig.~\ref{fig:G}. We highlight $A$ and $|G|^+_{A}$ for various $A$. We omitted the $A_x$, $A\ne\varnothing$, $x\not\in A$ for simplification.}\label{fig:H}
\end{center}
\end{figure}

\begin{definition}
    Let $X$ be a local order. We set $\Tilde{X}:=\sqcup_{x\in X}\{A_x\in\oindexp{X,x}{n}\,|\,A\ni x\}\subseteq B(\oindexp{X}{n})$ (with the induced topology and local order), and we call it the \emph{$n$-blowup} of $X$. The restriction of $\beta^+_X$ to $\Tilde{X}$ is noted $\beta_X$ and called the \emph{$n$-blowup map}. The restriction of the ordered base $(U^+_A)_{U,A}$ is noted $(U_A)_{U,A}$.
\end{definition}

\noindent We will omit the dimension $n$ when it is clear from the context. Note that in particular, $\Tilde{X}$ is a euclidean local order: any point belongs to some open set $V_A$ locally isomorphic to an open subset of $\R^n$, namely $A$.

\begin{lemma}
    With the same notations, for any ordered open $U_A$, $(\beta_X)_{|U_A}$ is a dihomeomorphism on its image $A$. In particular, $\beta_X$ is a local embedding, and $\Tilde{X}$ is euclidean.
\end{lemma}
\begin{proof}
    The inverse is given by $x\mapsto A_x$, which is continuous by Proposition~\ref{description_etale}, since it is the restriction of a continuous map. The orders correspond by definition.
\end{proof}

\begin{theorem}\label{blowup}
    Let $X$ be any local order. Then its $n$-blowup $\Tilde{X}$ satisfies the following universal property: for any euclidean local order $E$ of dimension $n$, for any local embedding $f:E\xrightarrow{}X$, there is a unique continuous map $\Tilde{f}:E\xrightarrow{}\Tilde{X}$ making the following diagram commute 
    \begin{center}
        \begin{tikzcd}
        & \Tilde{X}\arrow[d,"\beta_X"]\\
        E \arrow[r,"f"'] \arrow[ur, "\Tilde{f}", dashed] & X
    \end{tikzcd}
    \end{center}
    Moreover, $\Tilde{f}$ is a local embedding. 
\end{theorem}
\begin{proof}
    First we prove uniqueness. Suppose that $\Tilde{f}$ exists. Take $x\in E$. Because of the commutation of the above diagram, $\Tilde{f}(x)$ is of the form $A_{f(x)}\in\oindexp{X,f(x)}{n}$ with $f(x)\in A$. Since $\Tilde{f}(x)$ corresponds to the equivalence class of a subspace $A$ of an ordered open $V$ of $X$, $V_A$ is an ordered neighborhood of $\Tilde{f}(x)$. Since $\Tilde{f}$ is continuous, there is an ordered neighborhood $U$ of $E$, that we can assume isomorphic to $\R^n$, such that $\Tilde{f}(U)\subseteq V_A$. Composing with $\beta_X$, we get that $f(U)\subseteq A$. 
    \begin{claim}
        $A\sim_{f(x)} f(U)$.
    \end{claim}
    \begin{proof}
        $V_{f(U)}$ is open in $V_A$ because they are both open in $\tilde{X}$, and $\beta_X$ restricted to $V_A$ is a homeomorphism on its image $A$, so $f(U)$ open in $A$: there is an open set $W$ of $X$ such that $f(U)=A\cap W$.
    \end{proof}
    \noindent So $\Tilde{f}(x)$ is necessarily the class of $f(U)$. This proves uniqueness. Now for existence, just define $\Tilde{f}$ as suggested. More precisely, since $f$ is a local embedding, we can restrict to an open subset $U$ of $E$ where it is an isomorphism (of ordered spaces, but we don't really need that) on its image. Then we define $\Tilde{f}$ on $U$ by $x\mapsto f(U)_{f(x)}$. Now this is well-defined on all $E$ because if $U$ and $V$ are two such open subsets with $x\in U\cap V$, then $f(U)\sim_{f(x)} f(V)$ if $U$ and $V$ are neighborhoods of $x$, because they are both equivalent to $f(U\cap V)$. Then Proposition~\ref{description_etale} and the fact that $f$ is a local embedding automatically gives that $\Tilde{f}$ is a local embedding.
\end{proof}

\noindent In fact, we can study the functoriality of the blowup in general. 

\begin{proposition}\label{functoriality}
    Let $f:X\xrightarrow{}Y$ be any map between local orders. The following are equivalent:
    \begin{enumerate}
        \item There is a unique continuous map $\Tilde{f}:\Tilde{X}\xrightarrow{}\Tilde{Y}$ making the following square commute.
        \begin{center}
            \begin{tikzcd}
                \Tilde{X}\arrow[r, "\Tilde{f}"]\arrow[d,"\beta_X"']&\Tilde{Y}\arrow[d, "\beta_Y"]\\
                X\arrow[r, "f"']&Y
            \end{tikzcd}
        \end{center}
        Moreover, $\Tilde{f}$ is a local embedding.
        \item There is a local embedding $\Tilde{f}:\Tilde{X}\xrightarrow{}\Tilde{Y}$ making the above square commute
        \item For every local embedding $g:E\xrightarrow{}X$ with $E$ euclidean, $f\circ g$ is a local embedding. 
    \end{enumerate}
\end{proposition}
\begin{proof}
    Clearly (i) implies (ii). Now suppose (ii), and take a local embedding $g:E\xrightarrow{}X$. It admits a lift $\Tilde{g}:E\xrightarrow{}\Tilde{X}$ which is a local embedding, so $f\circ g=\beta_Y\circ\Tilde{f}\circ \Tilde{g}$ is a local embedding, so (ii) implies (iii). Finally, if we suppose (iii) then $f\circ \beta_X$ is a local embedding and $\Tilde{X}$ is euclidean, so there is a unique continuous lift which is moreover a local embedding, by Theorem~\ref{blowup}. 
\end{proof}

\noindent A morphism of local orders satisfying one of the above conditions is called \emph{weakly euclidean}. Notice that the uniqueness part of (i) shows that $X\mapsto\Tilde{X}$ is functorial from the category of local orders and weakly euclidean morphisms to the category of euclidean local orders and local embeddings. Now we finally turn to the proof of the announced adjunction.

\begin{lemma}\label{we_euclidean}
    Let $f:E\xrightarrow{}X$ be a morphism of local orders with $E$ euclidean. Then $f$ is weakly euclidean if and only if it is a local embedding.
\end{lemma}
\begin{proof}
    A local embedding is clearly weakly euclidean. Conversely, apply the characterization (iii) of Proposition~\ref{functoriality} to the identity $E\xrightarrow{}E$.
\end{proof}

\begin{corollary}
    The category $\mathcal{E}_n$ of euclidean local orders and local embeddings (of dimension $n$) is a coreflective subcategory of the category of local orders and weakly euclidean morphisms. In addition, $\mathcal{E}_n$ is coreflective in $\mathbb{L}_e$.
\end{corollary}
\begin{proof}
    For the first part, note that $\mathcal{E}_n$ is a full subcategory by Lemma~\ref{we_euclidean}. In both cases, the right adjoint is given by $X\mapsto \Tilde{X}$, by Theorem~\ref{blowup} and Proposition~\ref{functoriality}.
\end{proof}

\subsection{Locally strongly connected spaces}
We can weaken the definition of weakly euclidean morphism if we restrict to a subcategory of locally ordered spaces to \emph{locally strongly connected spaces}. This is interesting because typical examples of locally strongly connected spaces are given by locally ordered realizations of precubical sets, essentially by definition. 

\begin{definition}
    An ordered space $X$ is called \emph{strongly connected} if for all $x,\,y\in X$, $x\leq y$ if and only if there is a directed path $\gamma:\R\xrightarrow{}X$ and $a,\,b\in\R$ with $a\leq b$ and $\gamma(a)=x,\,\gamma(b)=y$. A local order $X$ is said to be \emph{locally strongly connected} if $X=\cup_i U_i$ with $U_i$ ordered strongly connected spaces.
\end{definition}

\noindent Note that in the definition of strongly connected ordered space, one direction is superfluous. Indeed, take $x,y$ in $X$ ordered space, $\gamma:\R\xrightarrow{}X$ directed path and $a,\,b\in\R$ with $a\leq b$ and $\gamma(a)=x,\,\gamma(b)=y$. We have $[a,b]$ compact, so we can cover it with open intervals $(]a_i,b_i[)_{i\in I}$, $I$ finite and $\gamma$ order preserving on each $]a_i,b_i[$. We can suppose that $I=\{0,\dots,n\}$ and $i\leq j$ implies $a_i\leq a_j$, and that $]a_i,b_i[\cap [a,b]\ne\varnothing$ for all $i$. An easy induction then shows that $\gamma(a)\leq\gamma(a_i)$ for all $i>0$ (two consecutive intervals intersect because the union of all the intervals covers $[a,b]$). Then we conclude that $\gamma(a)\leq\gamma(b)$ because $b\in]a_n,b_n[$. 

\begin{lemma}\label{we_mor}
    Let $X$ and $Y$ be local orders with $X$ locally strongly connected, $f:X\xrightarrow{}Y$ a continuous map. If $f$ satisfies one of the equivalent conditions of Proposition~\ref{functoriality}, then $f$ is a morphism of locally ordered spaces (so it is weakly euclidean).
\end{lemma}
\begin{proof}
    Let $V$ an ordered open of $Y$. $U:=f^{-1}(V)$ is open in $X$, and up to restriction we can suppose that $U$ is strongly connected. Now take $x,\, y\in X$, $x\leq y$. Take $\gamma:\R\xrightarrow{}U$ a directed path and $a,\,b\in\dom (\gamma)$ with $a\leq b$ and $\gamma(a)=x,\,\gamma(b)=y$. Note that we can suppose that $\gamma$ is a local embedding, by collapsing the segments where it is locally constant. But $f$ satisfies 3 of Proposition~\ref{functoriality}, so $f\circ\gamma:\R\xrightarrow{}V$ is a local embedding, in particular it is a directed path, so $f(x)\leq f(y)$.
\end{proof}

\section{A combinatorial description in the case of realizations of precubical sets}\label{sec_pset}

We can give a combinatorial characterization of the embeddings of $\R^n$ into the locally ordered realization of a precubical set. In this section, we suppose that all precubical sets have a locally ordered realization. The first step of the proof is a lemma which is interesting on its own because it can be seen as local variant of the well-known Theorem of invariance of domain~\cite[Theorem 2B.3]{hatcher2005algebraic}. At the end of the section, we give a consequence of the lemmas we proved, showing that there is a unique precubical set whose locally ordered realization is isomorphic to $\R^n$, namely the `infinite ($n$-)grid'. However, we have to start with a technical tool, namely subdivisions of precubical sets. The proofs for the corresponding section~\ref{subdiv}, which are purely technical, can be found in Appendix~\ref{sec:proof-subdiv}.

\subsection{Subdivisions}\label{subdiv}
The \(s\)-\emph{subdivision} functor (for \(s\in\mathbb N\setminus\{0\}\)) is the unique monoidal realization functor which sends the graph \((\cdot\to\cdot)\) to the graph \(G_s=(\cdot\to\cdots\to\cdot)\) made of \(s\) arrows. 
Equivalently, it is the unique realization functor which sends the $n$-cube, which is the \(n\)-fold tensor product \((\cdot\to\cdot)^{\otimes n}\), to \(G_s^{\otimes n}\). It will be convenient to represent the vertices of \(G_s\) as the fractions \(\tfrac is\) with \(i\in\{0,\ldots,s\}\) and its arrows as the fractions \(\tfrac {2i-1}{2s}\) with \(i\in\{1,\ldots,s\}\); indeed, the source and target maps are \(t\mapsto t-\tfrac1{2s}\) and \(t\mapsto t+\tfrac1{2s}\). Following this point of view, the elements of the subdivision \(\tfrac1sK\) of a precubical set \(K\) are the ordered pairs \((x,(t_1,\ldots,t_m))\) with \(t_j\in\{\tfrac1{2s},\tfrac1s,\ldots,\tfrac{s-1}s,\tfrac{2s-1}{2s}\}\) and \(\dim x=m\); the dimension of \((x,(t_1,\ldots,t_m))\) is the cardinality \(n\) of the set \(\{j\in\{1,\ldots,m\}\:|\:t_j\not\in\{\tfrac1s,\ldots,\tfrac{s-1}s\}\}\); the \((n-1)\)-face operators are defined by 
\[
\delta^{-}_{i,n}(x,(t_1,\ldots,t_m)) 
=
\left\{
\begin{array}{ll}
(x,(t_1,\ldots,t_j-\tfrac 1{2s},\ldots,t_m)) & \text{if \(t_j\not=\tfrac 1{2s}\)} \\[2mm]
(\delta^{-}_{j,m}x,(t_1,\ldots,\hat t_j,\ldots,t_m))  & \text{if \(t_j=\tfrac 1{2s}\)} 
\end{array}
\right.
\] 

\[
\delta^{+}_{i,n}(x,(t_1,\ldots,t_m)) 
=
\left\{
\begin{array}{ll}
(x,(t_1,\ldots,t_j+\tfrac1{2s},\ldots,t_m)) & \text{if \(t_j\not=\tfrac{2s-1}{2s}\)} \\[2mm]
(\delta^{+}_{j,m}x,(t_1,\ldots,\hat t_j,\ldots,t_m))  & \text{if \(t_j=\tfrac{2s-1}{2s}\)} 
\end{array}
\right.
\] 
with \(j\) being the index of the \(i^{\emph{th}}\) occurrence of an arrow of \(G_s\) in the \(m\)-tuple \((t_1,\ldots,t_m)\). For example, we have represented a simple square and its baryncentric (i.e.~$s=2$) subdivision $(\cdot\xrightarrow{}\cdot\xrightarrow{}\cdot)^{\otimes 2}$:
\begin{center}
  \begin{tikzpicture}
    \fill[mygray] (0,0) rectangle (3,3);
      \filldraw (0,0) circle (0.05);
      \filldraw (3,0) circle (0.05);
      \filldraw (0,3) circle (0.05);
      \filldraw (3,3) circle (0.05);
      \draw
      (0,0) edge[below,thick,"$e$"'] (3,0)
      (3,0) edge[below,thick] (3,3)
      (0,0) edge[below,thick] (0,3)
      (0,3) edge[below,thick] (3,3)
      ;
      \draw (1.5,1.5) node{$c$};
    \begin{scope}[xshift=45mm]
        \fill[mygray] (0,0) rectangle (3,3);
        \filldraw (0,0) circle (0.05);
        \filldraw (3,0) circle (0.05);
        \filldraw (0,3) circle (0.05);
        \filldraw (3,3) circle (0.05);
        \filldraw (1.5,0) circle (0.05);
        \filldraw (0,1.5) circle (0.05);
        \filldraw (3,1.5) circle (0.05);
        \filldraw (1.5,3) circle (0.05);
        \filldraw (1.5,1.5) circle (0.05);
        \draw
        (0,0) edge[below,thick,"\tiny ${(e,1/4)}$"'] (1.5,0)
        (1.5,0) edge[below,thick,"\tiny ${(e,3/4)}$"'] (3,0)
        (3,0) edge[below,thick] (3,3)
        (0,0) edge[below,thick] (0,3)
        (0,3) edge[below,thick] (3,3)

        (1.5,0) edge[below,thick] (1.5,1.5)
        (1.5,1.5) edge[below,thick] (1.5,3)
        (0,1.5) edge[below,thick] (1.5,1.5)
        (1.5,1.5) edge[below,thick] (3,1.5)
        ;
        \draw (1.5,1.5) node[above right]{\tiny ${(c,(1/2,1,2))}$};
        \draw (0.75,0.75) node{\tiny ${(c,(1/4,1,4))}$};
    \end{scope}
       
  \end{tikzpicture}
\end{center}
We can also define the \emph{underlying cube} $\pi_1(x,(t_1,\dots, t_m)):=x$ of $(x,(t_1,\dots, t_m))$. We give some useful properties of subdivisions.

\begin{restatable}{proposition}{lorderedsubdivisionrestate}
    \label{prop:locally_ordered_subdivision}
If \(s\) does not divide \(k\in\mathbb N\setminus\{0\}\), then for every element \(e\in\tfrac1sK\) of dimension \(n\geqslant k\), the set of lower \((n-k)\)-faces of \(e\) is disjoint from its set of upper \((n-k)\)-faces. 
\end{restatable}

\noindent As an immediate consequence we have: 
\begin{corollary}\label{cor:locally_ordered_subdivision} 
If \(K\) is of finite dimension \(d\) and if \(s>d\), then for every element \(x\) of \(\tfrac1sK\), the set of lower faces of \(x\) is disjoint from its set of upper faces.
\end{corollary}

\noindent With these tools in hand, we can prove the following, which justifies condition (iv) of \S\ref{sec:lorps}.

\begin{restatable}{lemma}{condrestatable}\label{cond4}
    Let $P$ be a precubical set of bounded dimension. Then for every vertex $v$ of $P$, there is an ordered neighborhood $B$ of $v$ only intersecting cubes adjacent to $v$, and for every cube $c$ with minimum (respectively maximum) $v$, for every point $x\in\{c\}\times]0,1[^{\dim(c)}$, every point $y\geq x$ (respectively $y\leq x$) in $B$ belongs to $\{c'\}\times]0,1[^{\dim(c')}$ for some cube $c'$ having $c$ as lower (respectively upper) face.
\end{restatable}

\noindent These properties are inherited by subset, so this is enough to prove that we can suppose that all the elements of the base have this property. The lemma should still be true without supposing that the precubical set is bounded in dimension, by using an explicit description of the ordered base of its locally ordered realization (which is complicated), but we only study bounded precubical sets in this paper. 

\begin{definition}
    Let $P$ be a precubical set, $v$ a vertex of $P$. 
    \begin{enumerate}
    \item The \emph{combinatorial neighborhood} $N(v)$ of $v$ is the set of cubes of $P$ which are adjacent to $v$.
    \item For a cube $c\in N(v)$, $v$ is \emph{only on one of the corners} of $c$ if the evaluation map from $0$-face operators to vertices of $c$ does not send two different elements to $v$.
    \item For a natural number $k$, $v$ is \emph{$k$-simple} if for every $k$-cube $c$ adjacent to $v$, $v$ is only on one of the corners of $c$.
    \end{enumerate}
\end{definition}

\begin{restatable}{lemma}{Wconnectedrestatable}\label{W_connected}
    Let $P$ be a precubical set of dimension $n$, $v$ a vertex of $P$ seen as a vertex of $\frac{1}{s}P$ for $s>1$. Then $v$ is $n$-simple.
\end{restatable}
    
\subsection{Combinatorial blowup}\label{sec:comb}
We turn to the proof of the generalization of the Theorem of invariance of domain. First we need some definitions.

\begin{definition}
    Let $P$ be a precubical set, $Z$ a subset of its geometric realization $Z\subseteq|P|$. The intersection of $P$ and $Z$, noted $P\cap Z$, is the precubical set spanned by all the cubes of $Z$ (which we recall are the cubes $c$ such that $\{c\}\times]0,1[^{\dim(c)}\cap Z\ne \varnothing$).
\end{definition}

Recall from~\cite{grandis2003cubical} that every face map $P(n+k)\to P(n)$ decomposes uniquely as 
\[\delta^{\epsilon_1}_{n+1,i_1}\circ\delta^{\epsilon_2}_{n+2,i_2}\circ\ldots\circ\delta^{\epsilon_k}_{n+k,i_k}\]
with $0\leq i_1<i_2<\ldots<i_k<n+k$. This allows us to define the following.

\begin{definition}\label{def:wepsilon}
    Let $P$ be a precubical set, $v$ be a vertex. For every $c\in P_n$ having $v$ as vertex, define the \emph{position} $p(v,c)$ of $v$ with respect to $c$ as the set of all tuples of the form $(p_1,\dots,p_n)$ with $p_i\in\{0,1\}$ such that $\delta_{1,0}^{p_1}\circ\dots\circ\delta_{n,n-1}^{p_n}(c)=v$. For $\epsilon\in]0,1[$, $p\in p(v,c)$, define 
    \[B(p,\epsilon):=\{x\in]0,1[^n\,|\,d_\infty(p,x)<\epsilon\}\]
    with $d_\infty$ the distance inherited from the infinite norm on $\R^n$. Finally, define the \emph{$\epsilon$-elementary open neighborhood of $v$} $W_\epsilon(v)$ as the subspace of the realization of $|P|$ given by 
    \[W_\epsilon(v):=\{v\}\cup \bigcup_{c\in N(v),p\in p(v,c)}\{c\}\times B(p,\epsilon)\]
\end{definition}

\noindent For example, we have highlighted $W_{1/2}(v)$ for the vertex $v$ of the following precubical set:
\vspace{1ex}
\begin{center}
    \begin{tikzpicture}
      \fill[mygray] (0,0) rectangle (2,2);
      \fill[gray, opacity=0.5] (1,0) rectangle (2,1);
      \filldraw (0,0) circle (0.05);
      \filldraw (2,0) circle (0.05);
      \filldraw (0,2) circle (0.05);
      \filldraw (2,2) circle (0.05);
      \filldraw (4,0) circle (0.05);
      \draw (2,0) node[below] {$v$};
      \draw
      (0,0) edge[below] (2,0)
      (2,0) edge[below] (2,2)
      (0,0) edge[below] (0,2)
      (0,2) edge[below] (2,2)
      (2,0) edge[below] (4,0)
      ;
      \draw[gray, line width=1.5mm,opacity=0.5] (2,0) -- (3,0);
      \draw[gray, line width=1.5mm,opacity=0.5] (2,0) -- (2,1);
      \draw[gray, line width=1.5mm,opacity=0.5] (1,0) -- (2,0);
    \end{tikzpicture}
\end{center}
Note that, with the same notations, if $v$ is $n$-simple then $|p(v,c)|\leq 1$ for every $n$-cube $c$ of $P$.

\begin{remark}\label{rem:wepsilon}
    With the same notations, note that $W_\epsilon(v)$ is indeed open in $|P|$. In fact, if $v$ is adjacent to a finite number of cubes, every open neighborhood of $v$ contains some $W_\epsilon(v)$ for a small enough $\epsilon$. More generally, for $x\in\{c\}\times]0,1[^{\dim(c)}$, a similar construction gives an open neighborhood of $x$ which only intersects cubes of higher dimension adjacent to $c$.
\end{remark}

\begin{lemma}[Invariance of domain]\label{invariance_domain}
    Let $P$ be a precubical set of dimension $n$, $v$ an $n$-simple vertex of $P$, $U$ an open neighborhood of $0$ in $\R^n$, $f:U\xrightarrow{}|P|$ an injective continuous map such that $f(0)=v$. Then there is a neighborhood $V\subseteq U$ of $0$ such that $f(V)$ is open in $|P\cap f(V)|$.
\end{lemma}
\begin{proof}
    Consider an elementary neighborhood $W_\epsilon(v)$ of $v$ in $|P|$. Note that the function $d_\infty(v,-)$ is well-defined and continuous on $W_\epsilon(v)$. Since $f$ is continuous, up to restriction we can suppose that $f(U)\subseteq W_\epsilon(v)$. We can also suppose without loss of generality that the (closed) unit ball (for the infinite norm) $B$ of $\R^n$ is included in $U$. Now $B$ and $S:=S^{n+1}\subseteq B$ are compact, so $f(B)$ and $f(S)$ are compact (and in particular closed because $W_\epsilon(v)$ is Hausdorff). So $d_\infty(v,-)_{|f(S)}$ is a continuous function defined on a compact which takes values in $\R$, so it is bounded and reaches its bounds. Note that the minimum of this function cannot be $0$, because otherwise $v=f(0)\in f(S)$ so $0\in S$ because $f$ is injective. Let $\eta>0$ be this minimum. We note $W:=W_{\eta/2}(v)$ and $V:=\mathring{B}\cap f^{-1}(W)$. 
    \begin{claim}\label{claim}
    $f(V)=|P\cap f(V)|\cap W$.
    \end{claim}
    \noindent The proof of the claim is topological routine and can be found in Appendix~\ref{sec:proof-invariance}. This concludes the proof of the lemma.
\end{proof}

\noindent For example, the embeddings in Fig.~\ref{fig:G} are not open, because an open neighborhood of the central vertex must intersect all the adjacent edges (recall Remark~\ref{rem:wepsilon}). But they become open after restriction to a subgraph (i.e.~after `removing' the edges which are not in the image). Note also that it is necessary to suppose that $v$ is simple. Indeed, consider the grey embedding on the left:
\begin{center}
    \begin{tikzpicture}
    \draw (-1,0)--(1,0);
    \draw (0,0.5) circle (0.5);
    \fill (0,0) circle (0.05);
    \draw[line width=1.5mm,opacity=0.3,line cap=round] (0,0.5) ++(270:0.5cm) arc (270:360:0.5cm);
    \draw[line width=1.5mm,opacity=0.3,line cap=round] (-0.5,0)--(0,0);

    \begin{scope}[xshift=35mm]
        \draw (-1,0)--(1,0);
    \draw (0,0.5) circle (0.5);
    \fill (0,0) circle (0.05);
    \fill (0,1) circle (0.05);
    \draw[line width=1.5mm,opacity=0.3,line cap=round] (0,0.5) ++(270:0.5cm) arc (270:360:0.5cm);
    \draw[line width=1.5mm,opacity=0.3,line cap=round] (-0.5,0)--(0,0);
    \end{scope}
\end{tikzpicture}
\end{center}
It is not open since an open neighborhood of the central vertex should intersect both ends of the loop, and it cannot be made open by restricting to a subgraph, precisely because there is an edge with same source and target. However, by Lemma~\ref{W_connected}, we can always deal with this problem by (restricting the domain and) subdividing the codomain as on the right: now we get an open embedding by removing one edge (the left one) of the loop.

On the other hand, we can characterize the `combinatorial embeddings' of $\R^n$ in some precubical set. First we give a rigorous definition capturing this intuition.

\begin{definition}\label{canonical}
    \emph{A local precubical structure (lps) of $\R^n$} is a pair $(P,v)$ where $P$ is a precubical set and $v$ a vertex of $P$, such that there exists an open neighborhood $W$ of $v$ in $|P|_{lo}$ included in an ordered neighborhood $U$ of $v$, and such that $W$ with the induced order is dihomeomorphic to $\R^n$ by a dihomeomorphism sending $0$ to $v$. If $(P,v)$ is a local precubical structure of $\R^n$, we also say that the precubical set $P$ is \emph{locally euclidean} at $v$.
\end{definition}

\noindent It turns out that we can give a purely combinatorial characterization of local precubical structures of $\R^n$. Note that the $n$-fold tensor product $(\cdot\xrightarrow{}\cdot\xrightarrow{}\cdot)^{\otimes n}$ with its central vertex is \textit{a priori} a good candidate to represent a combinatorial analog of the pointed space $(\R^n,0)$, modulo boundary issues which can be dealt with by subdivision. This is essentially the content of the next theorem, which is proved in Appendix~\ref{sec:proof-local}.

\begin{theorem}\label{explicit_blowup}
    Let $P$ be a precubical set, $v$ a vertex. The following are equivalent:
    \begin{enumerate}
        \item $(P,v)$ is a local precubical structure of $\R^n$;
        \item There is a precubical set $B\cong (\cdot\xrightarrow{}\cdot\xrightarrow{}\cdot)^{\otimes n}$ and a monomorphism 
        \[B\hookrightarrow \frac{1}{3}P\]
        sending the central vertex to $v$, and whose image is the precubical set generated by $N(v)$, with $v$ seen as a vertex of $\frac{1}{3}P$.
    \end{enumerate}
\end{theorem}

\begin{remark}\label{example}
\begin{enumerate}
    \item The proof of the last theorem heavily relies on the order structure, which is reasonable since the result is absolutely not true in the non-directed case.
    \item In the previous theorem, if $P$ is already a precubical subset of $\frac{1}{3}Q$ for some precubical set $Q$, there is no need to take the $3$-subdivision again, just because the proof still works.
    \item Let $R$ be the countable directed linear graph without endpoints, i.e.~$R_0=\Z$, $R_1=\Z$, $s=\id$, $t=(x\mapsto x+1)$. Then for any vertex $v$ of $R^{\otimes n}$, $(R^{\otimes n},v)$ is a local precubical structure of $\R^n$. In fact, this is the motivating example.
\end{enumerate}
\end{remark}

\noindent Now we can encapsulate the blowup of the locally ordered realization of a precubical set $P$ in a purely combinatorial presheaf. This is partly inspired by cellular sheaves~\cite{curry2014sheaves,ghrist2014elementary}. For every $c\in P_k$, we write $m(c)$ for the vertex of $\frac{1}{2}P$ which is the `center' of $c$ (with the notations of \S\ref{subdiv}, $m(c)=(c,(1/2,\dots,1/2))$). This allows to `represent' each cube of $P$ by a vertex. Recall~\cite[\S 12.2]{barr1990category} that the category of elements $\int P$ of a precubical set $P$ is the category whose objects are pairs $(k,c)$ with $k\in\N$, $c\in P_k$, and morphisms are of the form $f:(k',P(f)(c))\to(k,c)$ with $f:k'\to k$ a morphism of $\square$, i.e.~the objects are the cubes of $P$ and the morphisms are the face inclusions.

\begin{proposition}\label{prop:comb}
    Let $P$ be a precubical set. There is a presheaf $\mathsf{Comb}_P$ on $(\int P)^{op}$ satisfying for every $c\in P_k$, 
    \begin{align*}
        \mathsf{Comb}_P(k,c)=\{ & Q\subseteq \frac{1}{6}P\,|\,\text{$(Q,m(c))$ is an lps of $\R^n$} \\
        & \text{and }Q \cong (\cdot\xrightarrow{}\cdot\xrightarrow{}\cdot)^{\otimes n}\}\sqcup\{\bot^{(k,c)}\}
    \end{align*}
    where `$\subseteq$' means `subobject' and $\bot^{(k,c)}$ is a new element (noted $\bot$ when $(k,c)$ is clear from the context) satisfying $\mathsf{Comb}_P(f)(\bot)=\bot$ for any morphism $f$ of $\int P$.
\end{proposition}    
    
\noindent See Appendix~\ref{sec:proof-comb} for a proof. From this proof, we see that there is a canonical action on morphisms for $\mathsf{Comb}_P$ which is the combinatorial analog of restriction, and which we now illustrate. Take the graph $G$ of Fig.~\ref{fig:G}. Consider the local precubical structure of $\R$ in $\frac{1}{6}G$ (which essentially corresponds to the last traversal of Fig.~\ref{fig:G}) defined by $(B,v)$ below. 
\begin{center}
    \begin{tikzpicture}[scale=0.75]
            \draw (-3,0) -- (3,0);
            \draw (0,-3) -- (0,3);
            \filldraw (0,0) circle (0.075);
            \filldraw (0,1.5) circle (0.05);
            \filldraw (1.5,0) circle (0.05);
            \filldraw (0,-1.5) circle (0.05);
            \filldraw (-1.5,0) circle (0.05);
            \filldraw (3,0) circle (0.075);
            \filldraw (0,3) circle (0.075);
            \filldraw (-3,0) circle (0.075);
            \filldraw (0,-3) circle (0.075);
            \filldraw (0,0.5) circle (0.025);
            \filldraw (0.5,0) circle (0.025);
            \filldraw (0,-0.5) circle (0.025);
            \filldraw (-0.5,0) circle (0.025);
            \filldraw (1,0) circle (0.025);
            \filldraw (0,1) circle (0.025);
            \filldraw (-1,0) circle (0.025);
            \filldraw (0,-1) circle (0.025);
            \filldraw (2,0) circle (0.025);
            \filldraw (0,2) circle (0.025);
            \filldraw (-2,0) circle (0.025);
            \filldraw (0,-2) circle (0.025);
            \filldraw (2.5,0) circle (0.025);
            \filldraw (0,2.5) circle (0.025);
            \filldraw (-2.5,0) circle (0.025);
            \filldraw (0,-2.5) circle (0.025);

            \draw (0,0) node[above left] {$v$};
            \draw (0,0) node[below right] {$B$};
            \draw (1.5,0) node[above] {$m(e_1)$};
            \draw (1.5,0) node[below] {$B'$};
            \draw (0,1.5) node[above right] {$m(e_2)$};
            \draw (0.25,0) node[above] {$e$};

            \draw[gray, line width=1.5mm,opacity=0.5] (0,-0.5) -- (0,0);
            \draw[gray, line width=1.5mm,opacity=0.5] (0,0) -- (0.5,0);
            \draw[gray, line width=1.5mm,opacity=0.5] (1,0) -- (2,0);
    \end{tikzpicture}
\end{center}
There are morphisms $s_i:(0,v)\to(1,e_i)$, $i\in\{1,2\}$ in $\int G$. We look at the set $S_1$ of cubes of $B$ whose underlying cube in $G$ is the edge $e_1$ on the right. Notice that $S_1$ (only) contains the edge $e$ `going inside $e_1$', namely $e=(e_1,(1/12))$. Now the precubical set $B'$ obtained by `shifting $S_1$ to the center of $e_1$', namely the precubical set spanned by 
\[\{ (x, (t+1/6)), (x, (t+1/2))\,|\,(x, (t))\text{ is an edge of }S_1\} \]
is a local precubical structure of $\R^n$. We define $\mathsf{Comb}_G(s_1)(B):=B'$. On the contrary, the set $S_2$ of cubes of $B$ whose underlying cube in $G$ is the edge $e_2$ above is empty. We define $\mathsf{Comb}_G(s_2)(B):=\bot$, so $\bot$ is the combinatorial analog of the empty space. 
    
Another example, in higher dimension, is given in Fig.~\ref{e}. 
\begin{figure}[ht!]
\begin{center}
\begin{tikzpicture}
\fill[mygray] (-0.5,-0.5) rectangle (0.5,0.5);
\fill[mygray] (1,-0.5) rectangle (2,0.5);
\draw[very thick] (-3,-1.5) rectangle (3,1.5);
\draw[very thick] (0,-1.5) -- (0,1.5);
\draw (-3,0) -- (3,0);
\draw[densely dotted] (-3,-0.5) -- (3,-0.5);
\draw[densely dotted] (-3,-1) -- (3,-1);
\draw[densely dotted] (-3,0.5) -- (3,0.5);
\draw[densely dotted] (-3,1) -- (3,1);
\draw (-1.5,-1.5) -- (-1.5,1.5);
\draw (1.5,-1.5) -- (1.5,1.5);
\draw[densely dotted] (2.5,-1.5) -- (2.5,1.5);
\draw[densely dotted] (2,-1.5) -- (2,1.5);
\draw[densely dotted] (1,-1.5) -- (1,1.5);
\draw[densely dotted] (0.5,-1.5) -- (0.5,1.5);
\draw[densely dotted] (-2.5,-1.5) -- (-2.5,1.5);
\draw[densely dotted] (-2,-1.5) --   (-2,1.5);
\draw[densely dotted] (-1,-1.5) --   (-1,1.5);
\draw[densely dotted] (-0.5,-1.5) -- (-0.5,1.5);
\fill (0,0) circle (0.05);
\fill (-1.5,0) circle (0.05);
\fill (1.5,0) circle (0.05);
\draw (-1,0.25) node {\scalebox{0.8}{\(m(c_1)\)}};
\draw (0.5,0.25) node {\scalebox{0.8}{\(m(e)\)}};
\draw (2,0.25) node {\scalebox{0.8}{\(m(c_2)\)}};
\end{tikzpicture}
\end{center}
\caption{The central vertical edge $e$ is a common face of the cubes $c_1$ and $c_2$; the central grey cube is a local precubical structure of $\R^n$ centered at $m(e)$, and the right grey cube is its image under the action of $\mathsf{Comb}_P$, an lps of $\R^n$ centered at $m(c_2)$.}\label{e}
\end{figure} 

In the remaining of the paper, we assume that the action on morphisms of $\mathsf{Comb}_P$ is the canonical one. Now we can give a combinatorial analog of $\oindexp X n$.

\begin{definition}
    Let $P$ be a precubical set. An open set $U$ of $|P|$ is \emph{minimal} if there is a cube $\min(U)$ of $U$ such that $\min(U)$ is a face of every cube of $U$.
\end{definition}

\noindent Note that minimal open sets generate the topology by Remark~\ref{rem:wepsilon}.

    \begin{definition}
        Let $P$ be a precubical set of dimension $n$. Let $O_{\mathsf{min}}(|P|)$ be the full subcategory of $O(|P|)$ given by minimal open sets. For $U\in O_{\mathsf{min}}(|P|)$, define $\min(U)$ to be this `minimal' cube. Now we define a presheaf $C_P$ on $O_{\mathsf{min}}(|P|)$ by $C_P(U):=\mathsf{Comb}_P(\dim(\min(U)),\min(U))$. For $U\subseteq V$, if $U$ intersects $\min(V)$ then the restriction map is given by identity, else it is given by the the action of $\mathsf{Comb}_P$. We can extend this presheaf to $O(|P|)$ by right Kan extension, and the resulting presheaf is also called $C_P$.
    \end{definition}

    \begin{theorem}
        Let $P$ be a precubical set of dimension $n$, $B(C_P)$ the étale bundle associated to $C_P$. Then $B(C_P)\setminus\{\bot_x\,|\,x\in |P|\}$ is the $n$-blowup of $|P|_{lo}$.
    \end{theorem}
    \begin{proof}
        Define a morphism $r:C_P\xrightarrow{}\oindexp{X}{n}$ in the following way: for $U\in O_{\mathsf{min}}(|P|)$, send a local precubical structure of $\R^n$ $(Q,v)$ to the `maximal' subset of $U$ corresponding to it, that is the subset consisting of the realization of the underlying cubes of $Q$. One readily checks that $r$ is a natural transformation, i.e.~it is compatible with restriction. Since the open sets of $O_{\mathsf{min}}(|P|)$ form a base of $|P|$, this is enough to define $r$. Looking at the map induced by $r$ on the stalks, we see that $r$ induces a morphism $B(r):B(C_P)\xrightarrow{}B(\oindexp{|P|}{n})$ over $|P|$ which restricts to a morphism $B(C_P)\setminus\{\bot_x\,|\,x\in |P|\}\xrightarrow{}\Tilde{|P|}$. By Lemma~\ref{invariance_domain} (and Lemma~\ref{W_connected}) and Theorem~\ref{explicit_blowup} (and Remark~\ref{example} (2)), this map is bijective. But this is the restriction of an étale bundle (because $g\circ f$ and $g$ étale implies $f$ étale), so it is in fact an isomorphism of locally ordered spaces.
    \end{proof}

\noindent The proof is essentially sheaf-theoretic, which is quite satisfying. The `non-categorical' content of the proof is clearly identified and entirely contained in Theorem~\ref{explicit_blowup}. Note that this indeed gives a purely combinatorial description of the blowup, given a purely combinatorial description of the topology of $|P|$, which is indeed possible. Note also that we can build $B(C_P)$ without doing Kan extension first, which simplifies the construction. By way of an example, notice that all the examples of blowup we gave, for instance in Fig.~\ref{fig:G} and Fig.~\ref{fig:K}, already correspond to combinatorial blowups (where we do not subdivide because there is no need for it in these cases): the traversals we listed are always the `maximal' traversals, which correspond to combinatorial embeddings. So the last theorem is just the formalization of a natural intuition.

\subsection{Euclidean locally ordered realization}
    Recall the definition of $R^{\otimes n}$ from Remark~\ref{example} (3). One can wonder if the result of Theorem~\ref{explicit_blowup} can be generalized from the local setting to a more `global' one. The answer is positive, as fomalized by Theorem~\ref{global}. We sketch the proof below. Note that this theorem has a nice corollary: if $P$ is a subdivision of $Q$, then $P\cong R^{\otimes n}$ if and only if $Q\cong R^{\otimes n}$, since subdivision does not change the realization.  Recall that the $p$-cone of $\R^n$ for $p\in\{<,=,>\}^n$ is the subset $\{(x_1,\dots,x_n)\in\R^n\,|\,\forall i\, x_i p_i 0\}$. If the edges adjacent to a vertex $v$ are all sent to axis of $\R^n$, then we can naturally send the cubes adjacent to $v$ to the corresponding cones.

    \begin{theorem}\label{global}
			For every precubical set \(P\), we have 
			\[|P|_{lo}\cong(\R^n,\leqslant^n)\quad\Rightarrow\quad P\cong R^{\otimes n}\:.\]
    \end{theorem}
    \begin{proof}
        The idea is the following: using Theorem~\ref{explicit_blowup}, we know that $P$ is locally isomorphic to $R^{\otimes n}$. In addition, it is `infinitely unfolded' since $|P|\cong\R^n$, which will give the result.
        
        More precisely, fix an isomorphism $f:|P|\xrightarrow{}\R^n$. Choose a vertex $v$ of $P$. Without loss of generality, we suppose that $f(v)=0$. Then we can find a neighborhood $W$ of $v$ which satisfies the hypotheses of Theorem~\ref{explicit_blowup}, so $(P,v)$ is a local precubical structure of $\R^n$. As a consequence, we can define a map $\phi:W_{1/2}(v)\xrightarrow{}\R^n$ by sending $v$ to $0$, the edges to the axis and the cubes accordingly to the cones. We get that $W_{1/2}(v)$ is \emph{canonically isomorphic} to $\R^n$, in the sense that the preorder on it given by cubewise increasing paths is a partial order, and $\phi$ is a dihomeomorphism. Then, essentially by Lemma 5.4 of~\cite{haucourt2024non} and Corollary 10.4.10 of~\cite{schroder2003ordered}, and since $W_{1/2}(v)$ is canonically isomorphic to $\R^n$, $f$ sends $W_{1/2}(v)$ to a \emph{geometric cube} of $\R^n$ centered at $0$, where a geometric cube means a product of intervals, and it sends the edges in $W_{1/2}(v)$ to the axes of $\R^n$, and every cube to the corresponding cone. We can use this to prove that in fact $v$ is simple: if there is a cube $c$ with vertex $v$ in two different corners, then there are two distinct segments joining $v$ to the center of $c$, but the restriction of these segments to $W_{1/2}(v)$ are sent by $f$ to segments going in different directions, so by continuity they cannot meet at the image of the center of $c$ (picture below). 
        
        \begin{center}       
            \scalebox{1}{
            \begin{tikzpicture}
            \draw (-1,-1) rectangle (1,1) ;
            \draw[line width=1mm,color=gray] (-1,0) -- (-1,-1) -- (0,-1) ;
            \draw[line width=1mm,color=lightgray] (1,0) -- (1,1) -- (0,1) ;
            \draw[dashed] (-2,-2) rectangle (0,0) ;
            \draw[dashed] (0,0) rectangle (2,2) ;
            \filldraw[fill=white] (-1,-1) node[below]{\(v\)} circle (0.1) ;
            \filldraw[fill=white] (1,1) node[above]{\(v\)} circle (0.1) ;
            \draw (2.75,0) node{\(\to\)} node[above]{\(f\)} ;
            \begin{scope}[xshift=-0.5cm]
            \draw[line width=1mm,color=gray] (5,1) -- (5,0) -- (6,0) ;
            \draw[line width=1mm,color=lightgray] (4,0) -- (5,0) -- (5,-1);
            \filldraw[fill=white] (5,0) node[above right]{\(0\)} circle (0.1) ;
            \draw[dashed] (4,-1) rectangle (6,1);
            \foreach \x in {2,3,...,25} {
            \draw (1/\x+1/2,1/\x) circle (0.005);	
            \draw (-1/\x+1/2,-1/\x) circle (0.005);	
            \draw (6-1/\x,1-1/\x) circle (0.005);	
            \draw (4+1/\x,-1+1/\x) circle (0.005);	
            }
            \end{scope}
            \end{tikzpicture}
            }
        \end{center}
        
        \noindent As a consequence $W_1(v)$ is also canonically isomorphic to $\R^n$. By a similar argument, all the neighboring vertices of $v$ (graph-theoretically) are distinct. In fact, this defines a monomorphism (of symmetric precubical sets) $\iota:(-1\xrightarrow{}0\xrightarrow{}1)^{\otimes n}\xrightarrow{} P$ centered at $v$, such that the image of a tuple whose only nonzero term is in position $i$ is the (graph-theoretic) neighbor of $v$ which is sent to the $i$th axis of $\R^n$. Note that all the cubes that are adjacent to $v$ are in the image of $\iota$, because $(P,v)$ is a local precubical structure of $\R^n$. Note also that $\Im(f\circ|\iota|)$ is a geometric cube in $\R^n$ centered at $0$.

        
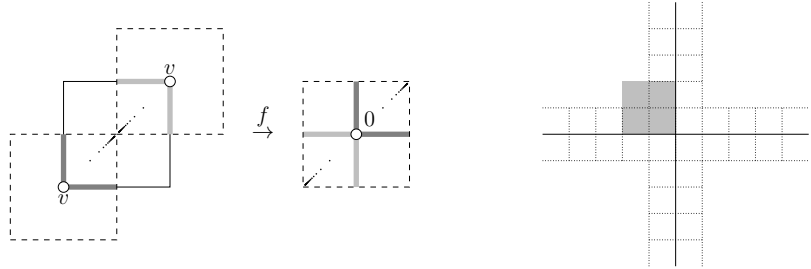
\begin{center}       
            \begin{tikzpicture}[scale=0.9]
            \fill[lightgray] (-1/2,0) rectangle (1/2,1) ;
            \foreach \x in {-3,-2,...,5}
            \draw[densely dotted] (\x/2,-1/2) -- (\x/2,1/2);
            \draw[densely dotted] (-2,1/2) -- (3,1/2) ;
            \draw[densely dotted] (-2,-1/2) -- (3,-1/2) ;
            \foreach \y in {-4,-3,-2,2,3,4} {
            \draw[densely dotted] (0,\y/2) -- (1,\y/2);	
            }
            \draw[densely dotted] (0,-5/2) -- (0,-1/2);
            \draw[densely dotted] (1,-5/2) -- (1,-1/2);
            \draw[densely dotted] (0,5/2) -- (0,1/2);
            \draw[densely dotted] (1,5/2) -- (1,1/2);
            \draw (-2,0) -- (3,0) ;
            \draw (1/2,-5/2) -- (1/2,5/2) ;
            \end{tikzpicture}
        \end{center}

        Now we want to define a family $\iota_p:(-1\xrightarrow{}0\xrightarrow{}1)^{\otimes n}\hookrightarrow P$ for every $p\in \Z^n$ which cover $P$. We already defined $\iota_{0\dots 0}:=\iota$. We consider every neighbor (in the graph-theoretic sense) of $v$ and do the same as in the last paragraph to get $\iota_p$ for $p$ a neighbor of $(0,\dots,0)$ in $\Z^n$. We do this inductively to get $\iota_{0\dots0k0\dots0}$ where $k$ is in position $i$, for every $1\leq i\leq n$ and $k\in\Z$. Simultaneously, we define sequences $(x_{i,j})_{j\in\Z}$ for $1\leq i\leq n$, such that $x_{i,j}$ belongs to the $i$th axis of $\R^n$, and $\Im(f\circ|\iota|)$ is centered at $x_{i,j}$. In the above picture, we represented the axis of $\R^2$; the points on the axis are the $x_{i,j}$, and the dotted squares are the images of the $\iota$s. The important point is that these form the image of a precubical subset of $P$. Now for every $p\in \Z^n$, we build $\iota_p:(-1\xrightarrow{}0\xrightarrow{}1)^{\otimes n}\hookrightarrow P$ centered at $([x_{1,p_1}],\dots,[x_{n,p_n}])$, where $[x_{i,j}]$ is the value of the only nonzero coordinate of $x_{i,j}$, and whose other vertices (in $\R^n$) are the expected ones with respect to the $(p_i)_i$. We do this by induction, using the construction of the last paragraph and the fact that the image we get in $\R^n$ each time is a geometric cube, i.e.~if we know 3/4 of the cube we know all the cube. More precisely, we use the lexicographic order on each cone of $\Z^n$ (discarding the sign). For example, see the grey cube on the picture, which is the first inductive step in the left upper cone. Now notice that, by construction, if one vertex of $P$ belongs to the image of one of the $\iota_p$, all of its adjacent cubes also are, so in fact this is the case of all the cubes of $P$, because $P$ is connected (since its realization is isomorphic to $\R^n$). Putting everything together, we get an epimorphism 
        \[\bigsqcup_{p\in \Z^n}\iota_p:\bigsqcup_{p\in \Z^n}(-1\xrightarrow{}0\xrightarrow{}1)^{\otimes n}\xrightarrow{}P\]
        which exhibits $P$ as a quotient of $\bigsqcup_{p\in \Z^n}(-1\xrightarrow{}0\xrightarrow{}1)^{\otimes n}$ by an equivalence relation $S$ (since we are in a topos), whose realization exhibits $\R^n$ as a quotient of $\bigsqcup_{p\in \Z^n}|-1\xrightarrow{}0\xrightarrow{}1|^n$. This implies that 
        \[P\cong \bigsqcup_{p\in \Z^n}(-1\xrightarrow{}0\xrightarrow{}1)^{\otimes n}/S\cong R^{\otimes n}\]
    \end{proof}

\section{The path lifting property}\label{sec_path}
We have to mention that the $n$-blowup may forget a lot of points. 
For example, the $(n+1)$-blowup of the \(n\)-dimensional euclidean local order is empty. We may consider changing the universal property we want to achieve, for example lifting a larger class of morphisms, and adapt the construction to see what this gives. In this section, we discuss the directed paths lifting property, which is a very important property since directed paths represent execution traces of programs, and we do not want to forget any of these if we want to use our blowup in the context of concurrency theory. We give an idea of the way one can express this property with the sheaf structure. However, the results are still preliminary, and this section should be seen as a future work discussion. In the following, a directed path is supposed injective. This is not very restrictive knowing~\cite[Theorem 3.6]{fahrenberg2007reparametrizations}.

Notice that for every $n\in\N\setminus\{0\}$, $z\in \oindexp{\R^n,0}{1}\setminus\{A_0\,|\,0\not\in A\}$, for every local order $X$ and $x\in X$, there is a map \[z^*:\oindexp{X,x}{n}\setminus\{A_x\,|\,x\not\in A\}\xrightarrow{}\oindexp{X,x}{1}\setminus\{A_x\,|\,x\not\in A\}\]
defined in the following way: an element of $\oindexp{X,x}{n}\setminus\{A_x\,|\,x\not\in A\}$ is an equivalence class of order embeddings from an open subset of $\R^n$ to an ordered open of $X$. Notice that we can always take a representative $h:U\xrightarrow{}V$ of the class such that $U$ contains $0$ and that $h(0)=x$. Let $h$ be such a representative of an element of $\oindexp{X,x}{n}\setminus\{A_x\,|\,x\not\in A\}$, and $\gamma$ be such a representative of $z$. Then $z^*([h]):=[h\circ\gamma]$. This is clearly well-defined. See the picture below.

\begin{center}
    \begin{tikzpicture}
        \draw (-1,0) -- (1,0);
        \draw (0,-1) -- (0,1);
        \draw (0,0) node[above right] {$0$};
        \fill (0,0) circle (0.05);
        \draw (-1,-0.4) .. controls (-0.8,-0.35) and (-0.6,-0.3) .. 
    (-0.5,-0.25)   
    .. controls (-0.4,-0.2) and (-0.3,-0.15) .. 
    (-0.2,-0.1)    
    .. controls (-0.1,-0.05) and (0,0) .. 
    (0.2,0.05)    
    .. controls (0.4,0.15) and (0.6,0.25) .. 
    (0.7,0.3)     
    .. controls (0.8,0.35) and (0.9,0.45) .. 
    (1,0.5);
    \draw (0.9,0.6) node {$\gamma$};
        
        \begin{scope}[xshift=20mm]
            \draw (0,.5) node {$h$};    
            \draw (0,0) node {$\to$};          
        \end{scope}

        \begin{scope}[xshift=40mm]
            \draw (0,0) node[above left] {$v$};
            \fill (0,0) circle (0.05);
            \draw[smooth] 
        (-1,0) .. controls (-0.8,0.5) and (-0.5,0.7) .. (0,0.8)
               .. controls (0.3,0.9) and (0.6,0.5) .. (0.8,0.2)
               .. controls (0.9,-0.1) and (0.7,-0.6) .. (0.5,-0.8)
               .. controls (0.3,-0.9) and (-0.2,-0.7) .. (-0.4,-0.4)
               .. controls (-0.6,-0.2) and (-0.9,0.1) .. (-1,0);
            \draw (-0.5,-0.25)   
    .. controls (-0.4,-0.2) and (-0.3,-0.15) .. 
    (-0.2,-0.1)    
    .. controls (-0.1,-0.05) and (0,0) .. 
    (0.2,0.05)    
    .. controls (0.4,0.15) and (0.6,0.25) .. 
    (0.7,0.3);
    \draw (1,0.6) node {$[h]$};
    \draw (0.6,0) node {$z^*[h]$};

        \end{scope}
    \end{tikzpicture}
\end{center}

\begin{proposition}\label{local_dipath_lift}
    Let $X$ be a local order. The following are equivalent: 
    \begin{enumerate}
        \item The family $\{z^*\,|\,z\in \oindexp{\R^n,0}{\emph{1}}\setminus\{A_0\,|\,0\not\in A\}\}$ is jointly surjective for every $x\in X$.
        \item For every directed path $\gamma$ on $X$, every $x\in \dom(\gamma)$, there is an open neighborhood $U$ of $x$ such that $\gamma_{|U}$ admits a lift to $\Tilde{X}$ the $n$-blowup of $X$.
    \end{enumerate}
\end{proposition}

\noindent The proof is routine and given in Appendix~\ref{sec:proof-lift}. If one of the previous conditions is satisfied, we say that $X$ satisfies the \emph{local directed paths lifting property}. This implies the existence of a global lift for every directed path if, for example, the set of singularities of $X$ is discrete. Note that this is not true in general. For example, one can consider the subspace $\{(x,y,0)\in\R^3\,|\,x>-1\}\cup\{(x,0,z)\in\R^3\,|\,x<1\}$ of $\R^3$, which satisfies the local directed paths lifting property, but such that the directed path $x\in\R\mapsto(x,0,0)$ cannot be lifted to a (continuous) path on the corresponding blowup 
However, we can give a global version. For every $n\in\N\setminus\{0\}$, for every local order $X$, for every $[f:E\xrightarrow{}A]\in \oindexp{X}{n}(X)$ and $[\gamma:E'\xrightarrow{}E]\in \oindexp{E}{1}(E)$, we have an element $[f\circ\gamma]$ of $\oindexp{X}{1}(X)$. This defines a map $l_n:\bigsqcup_{[E\xrightarrow{}A]\in \oindexp{X}{n}(X)}\oindexp{E}{1}(E)\xrightarrow{}\oindexp{X}{1}(X)$. A similar proof to the previous one then gives:

\begin{proposition}
    For any local order $X$, every directed path on $X$ can be lifted to a directed path on $\Tilde{X}$ the $n$-blowup of $X$ if and only if the function $l_n$ defined above is surjective.
\end{proposition}

\noindent This is clearly always true for $n=1$, which explains why we always get the directed paths lifting property for graphs. Combinatorial versions of these characterizations is the case of realizations of precubical sets should be studied. In fact, it would be interesting to look at lifting properties \emph{up to (directed) homotopy}, i.e.~given a directed path, is there a dihomotopic path~\cite{grandis2009directed} which can be lifted to the blowup. This would be particularly interesting in the case of \emph{non-selfintersecting} precubical sets, since a directed path on their realization is always dihomotopic to a path on their $1$-skeleton~\cite{fajstrup2005dipaths}, which should give a simple criterion for the homotopy path lifting property. This is work in progress.

\section{Conclusion}
In this paper, we use the abstract theory of sheaves to generalize a blowup construction on all locally ordered spaces. This is interesting for several reasons. Firstly, sheaf theory significantly simplifies the proofs: notice that we got Proposition~\ref{description_etale} for free; in fact, one can look at the amount of work needed in~\cite{haucourt2024non} to prove Theorem~\ref{thm:blowup_graphs}, and compare it to our proof of the more general Theorem~\ref{blowup} (however, the proof from~\cite{haucourt2024non} also provides an explicit description of the blowup, which also needed a lot of work in \S\ref{sec_pset}). Secondly, doing the link with sheaf theory is interesting on its own right, and paves the way to generalizations of this procedure outside of directed algebraic topology. In fact, the first author is working on a generalization of the blowup in an abstract setting, which can be particularized to various contexts, including differential geometry and sequential spaces, and with possible applications to logic~\cite{de2023geometric,toen2008algebrisation}. One can also wonder whether the sheaf $\oindexp{X}{1}$ for some local order $X$ represents some kind of `ordered tangent bundle' of $X$, since its stalk at $x\in X$ precisely represents the equivalence classes of directed path crossing $x$. In addition, the idea of `completing' the blowup into an étale bundle could be interesting on its own, as it replaces a map with a much better behaved one. Finally, would also be inclined to adapt the sheaf theoretic blowup procedure to other models of directed spaces, such as streams~\cite{krishnan2009convenient} which already have a sheaf theoretic flavour.

More concretely, we wonder if the results of~\cite[\S6]{haucourt2024non} can be extended to our setting; in other words, if there is a pseudometric on the blowup of the locally ordered realization of a higher dimensional automaton \(A\) which reflects the execution time of the traces on \(A\). In~\cite{haucourt2024non}, this pseudometric comes from a \emph{smooth} manifold structure on the blowup. In this paper, we have only defined a locally ordered manifold structure on the blowup, but as suggested in the introduction, it should not be very hard to get a smooth atlas in the case of locally ordered realizations of precubical sets (we even conjecture the existence of an atlas whose related transition maps are identities). This would provide concrete applications of the generalized blowup to the study of conccurent systems. However, this is a bit far from the core subject of the present paper, and will be developed in a future work.



\bibliographystyle{alpha} 
\bibliography{b}

\appendix

\section{Subdivisions}\label{sec:proof-subdiv}
\input{subdiv}

\section{Proof of Claim~\ref{claim}}\label{sec:proof-invariance}
\begin{proof}
   \input{proof_invariance} 
\end{proof}

\section{Proof of Theorem~\ref{explicit_blowup}}\label{sec:proof-local}
\input{proof_local}

\section{Proof of Proposition~\ref{prop:comb}}\label{sec:proof-comb}
\input{proof_comb}

\section{Proof of Proposition~\ref{local_dipath_lift}}\label{sec:proof-lift}
\begin{proof}
Suppose (i) is true, and take $\gamma$ and $x$ as in (ii). By (i), there is $[A]\in \oindexp{X,x}{n}$ and $z\in \oindexp{\R^n,0}{1}$ such that $z^*([A])= (\Im{\gamma})_{\gamma(x)}$. Now this means that there is an open set $V\subseteq X$ containing $\gamma(x)$ such that $A\cap V=\Im{\gamma}\cap V$. So $x\mapsto A_{\gamma(x)}$, which a morphism of local orders because it is the composite of two such, is a lift of $\gamma$ on $\gamma^{-1}(V)\ne\varnothing$, which is what we wanted. 

Conversely, let $[a]\in \oindexp{X,x}{1}\setminus\{A_x\,|\,x\not\in A\}$. Take $\gamma$ a representative of $a$, which we can take to be a directed path $\R\xrightarrow{}X$ with $\gamma(0)=x$. By (ii) we can lift it to $\Tilde{X}$ (up to restriction), say by $\Tilde{\gamma}$, then we can take an open neighborhood $A$ of $\Tilde{\gamma}(0)$. Then $U:=\Tilde{\gamma}^{-1}(A)$ is an open subset of $\dom(\Tilde{\gamma})$, so also of $\dom(\gamma)$. $\Tilde{X}$ is locally euclidean, so we can suppose that $A$ is isomorphic to an open subset of $\R^n$. Now (up to translation) composing this isomorphism with $\Tilde{\gamma}$ gives an element $z$ of $\oindexp{\R^n,0}{1}\setminus\{A_0\,|\,0\not\in A\}$, and we clearly have $z^*(A_x)=[a]$, which concludes.
\end{proof}

\end{document}

%% file: introduction.tex
\section{Introduction}

The theoretical developments contained in this article are motivated by a computer science question, namely the modeling of concurrent programs. The purpose of this introduction is to clarify the context. 

\paragraph{Parallel composition of processes.} 
A \emph{concurrent program} is a \emph{parallel composition} (the operator is denoted by the vertical bar \(|\:\)) of \emph{processes} (\ie instances of sequential programs) possibly containing instructions that force synchronization (\emph{e.g.} the \emph{wait} instruction \texttt{W}) or prevent simultaneous execution (\emph{e.g.} the \emph{lock}/\emph{free} instructions \texttt{P}/\texttt{V}). 
Such kind of instructions appear in most of modern high-level languages (\emph{e.g.} Rust, Go, Java, OCaml) and assemblers (\emph{e.g.} x86-64, AArch64, RISC-V). 
As a common thread, we will study the programs
\begin{equation}\label{eqn:programs}
\texttt{P(m);V(m)\:|\:P(m);V(m)}\quad\text{and}\quad\texttt{W(b)\:|\:W(b)}\:. 
\end{equation}
In the first case, the processes try to acquire (\texttt{P(m)}) and release (\texttt{V(m)}) a resource \texttt{m}; in the second one, they wait \(\texttt{W(b)}\) behind a barrier \texttt{b} which forces them to synchronize.

\paragraph{Higher dimensional automata.}
An \emph{automaton} on an alphabet \(\Sigma\) is a $\Sigma$-labelled finite graph, i.e.~a morphism \(f\) from a finite graph to the \(\Sigma\)-\emph{bouquet} (\ie the graph with a single vertex whose set of arrows is \(\Sigma\)); assuming that \(\Sigma\) is the set of available instructions, each arrow \(a\) represents the execution of \(f(a)\).
Processes are naturally interpreted by automata. 
Now graphs admit a higher dimensional generalization, namely \emph{precubical sets} (\S\ref{sec:precubical_sets}), where we still have vertices and edges but also $n$-cubes for any natural number $n$.
Given that graphs are \(1\)-dimensional precubical sets, one defines a \emph{higher dimensional \(\Sigma\)-automaton} (\emph{HDA} for short) as a labelled precubical set, i.e.~a morphism \(f\) from a finite precubical set to the \emph{higher dimensional \(\Sigma\)-bouquet} (\ie the precubical set whose \(n\)-dimensional elements are the words of length \(n\) on \(\Sigma\)). 
A cube \(a\) represents the simultaneous execution of instructions \(i_1,\ldots,i_n\) with \(n=\dim(a)\) and \(f(a)=(i_1,\ldots,i_n)\); such a tuple is called a \emph{multi-instruction}. 
The \emph{tensor product} \(\otimes\) of precubical sets (\S\ref{sec:precubical_sets}) extends to higher dimensional automata, thus providing the interpretation of the parallel composition operator: 
given the processes \(P_1,\ldots,P_n\), 
we have \(\interpret{P_1|\ldots|P_n}=\interpret{P_1}\otimes\ldots\otimes\interpret{P_n}\) with \(\interpret{P}\) denoting the interpretation of the program (\resp process) \(P\) as an HDA. 
For example
\begin{center}
\begin{tikzpicture}
\fill(0.2,0) circle (0.025);
\fill(0,0.2) circle (0.025);
\fill(0,1) circle (0.025);
\fill(1,0) circle (0.025);
\draw[->,>=stealth,shorten >=0.7mm,shorten <=0.7mm] (0,0.2) -- (0,1);
\draw[->,>=stealth,shorten >=0.7mm,shorten <=0.7mm] (0.2,0) -- (1,0);
\draw (-0.15,0.55) node{\scriptsize b};
\draw (0.55,-0.15) node{\scriptsize a};
\draw(0,0) node {\(\otimes\)};
\draw (1.8,0.5) node {\(=\)};
\begin{scope}[xshift=30mm]
\fill(0,0) circle (0.025);
\fill(0,1) circle (0.025);
\fill(1,0) circle (0.025);
\fill(1,1) circle (0.025);
\draw[->,>=stealth,shorten >=0.7mm,shorten <=0.7mm] (0,0) -- (0,1);
\draw[->,>=stealth,shorten >=0.7mm,shorten <=0.7mm] (0,0) -- (1,0);
\draw[->,>=stealth,shorten >=0.7mm,shorten <=0.7mm] (0,1) -- (1,1);
\draw[->,>=stealth,shorten >=0.7mm,shorten <=0.7mm] (1,0) -- (1,1);
\fill[mygray] (0.7mm,0.7mm) rectangle (9.3mm,9.3mm);
\draw (0.5,0.5) node{\scriptsize ab};
\draw (-0.15,0.5) node{\scriptsize b};
\draw (1.15,0.5) node{\scriptsize b};
\draw (0.5,-0.15) node{\scriptsize a};
\draw (0.5,1.15) node{\scriptsize a};
\end{scope}
\end{tikzpicture}
\end{center}
\paragraph{Atomicity}
Instructions like \(\texttt{P}\), \(\texttt{V}\), and \(\texttt{W}\) 
are supposed to be \emph{atomic} in the sense that they behave as if 
their execution time was null. 
Taking this specificity into account requires a little subtler 
interpretation: atomic instructions have to be distinguished from 
ordinary ones in the definition a bouquet.   
Formally we write \(\Sigma=\Sigma_0\sqcup\Sigma_1\) with \(\Sigma_0=\set{\text{atomic instructions}}\),
and define the higher dimensional \(\Sigma\)-bouquet as the precubical set whose \(n\)-dimensional elements are the words on $\Sigma$
containing exactly \(n\) elements in $\Sigma_1$.
In particular, vertices can be labelled by atomic instructions.
In this setting, the interpretation of our toy programs (\ref{eqn:programs}) are
\begin{center}
\begin{tikzpicture}[scale=0.8]
\fill(0,0) circle (0.04);
\fill(0,1) circle (0.04);
\fill(0,2) circle (0.04);
\fill(0,3) circle (0.04);
\fill(1,0) circle (0.04);
\fill(1,1) circle (0.04);
\fill(1,2) circle (0.04);
\fill(1,3) circle (0.04);
\fill(2,0) circle (0.04);
\fill(2,1) circle (0.04);
\fill(2,2) circle (0.04);
\fill(2,3) circle (0.04);
\fill(3,0) circle (0.04);
\fill(3,1) circle (0.04);
\fill(3,2) circle (0.04);
\fill(3,3) circle (0.04);
\draw[->,>=stealth,shorten >=0.7mm,shorten <=0.7mm] (0,0) -- (0,1);
\draw[->,>=stealth,shorten >=0.7mm,shorten <=0.7mm] (0,1) -- (0,2);
\draw[->,>=stealth,shorten >=0.7mm,shorten <=0.7mm] (0,2) -- (0,3);
\draw[->,>=stealth,shorten >=0.7mm,shorten <=0.7mm] (1,0) -- (1,1);
\draw[->,>=stealth,shorten >=0.7mm,shorten <=0.7mm] (1,1) -- (1,2);
\draw[->,>=stealth,shorten >=0.7mm,shorten <=0.7mm] (1,2) -- (1,3);
\draw[->,>=stealth,shorten >=0.7mm,shorten <=0.7mm] (2,0) -- (2,1);
\draw[->,>=stealth,shorten >=0.7mm,shorten <=0.7mm] (2,1) -- (2,2);
\draw[->,>=stealth,shorten >=0.7mm,shorten <=0.7mm] (2,2) -- (2,3);
\draw[->,>=stealth,shorten >=0.7mm,shorten <=0.7mm] (3,0) -- (3,1);
\draw[->,>=stealth,shorten >=0.7mm,shorten <=0.7mm] (3,1) -- (3,2);
\draw[->,>=stealth,shorten >=0.7mm,shorten <=0.7mm] (3,2) -- (3,3);
\draw[->,>=stealth,shorten >=0.7mm,shorten <=0.7mm] (0,0) -- (1,0);
\draw[->,>=stealth,shorten >=0.7mm,shorten <=0.7mm] (1,0) -- (2,0);
\draw[->,>=stealth,shorten >=0.7mm,shorten <=0.7mm] (2,0) -- (3,0);
\draw[->,>=stealth,shorten >=0.7mm,shorten <=0.7mm] (0,1) -- (1,1);
\draw[->,>=stealth,shorten >=0.7mm,shorten <=0.7mm] (1,1) -- (2,1);
\draw[->,>=stealth,shorten >=0.7mm,shorten <=0.7mm] (2,1) -- (3,1);
\draw[->,>=stealth,shorten >=0.7mm,shorten <=0.7mm] (0,2) -- (1,2);
\draw[->,>=stealth,shorten >=0.7mm,shorten <=0.7mm] (1,2) -- (2,2);
\draw[->,>=stealth,shorten >=0.7mm,shorten <=0.7mm] (2,2) -- (3,2);
\draw[->,>=stealth,shorten >=0.7mm,shorten <=0.7mm] (0,3) -- (1,3);
\draw[->,>=stealth,shorten >=0.7mm,shorten <=0.7mm] (1,3) -- (2,3);
\draw[->,>=stealth,shorten >=0.7mm,shorten <=0.7mm] (2,3) -- (3,3);

\draw[->,>=stealth,shorten >=0.7mm,shorten <=0.7mm] (0,0) -- (1,0);
\draw[->,>=stealth,shorten >=0.7mm,shorten <=0.7mm] (0,1) -- (1,1);
\draw[->,>=stealth,shorten >=0.7mm,shorten <=0.7mm] (1,0) -- (1,1);
\fill[mygray] (0.7mm,0.7mm) rectangle (9.3mm,9.3mm);
\fill[mygray] (10.7mm,0.7mm) rectangle (19.3mm,9.3mm);
\fill[mygray] (20.7mm,0.7mm) rectangle (29.3mm,9.3mm);
\fill[mygray] (0.7mm, 10.7mm) rectangle (9.3mm, 19.3mm);
\fill[mygray] (10.7mm,10.7mm) rectangle (19.3mm,19.3mm);
\fill[mygray] (20.7mm,10.7mm) rectangle (29.3mm,19.3mm);
\fill[mygray] (0.7mm, 20.7mm) rectangle (9.3mm, 29.3mm);
\fill[mygray] (10.7mm,20.7mm) rectangle (19.3mm,29.3mm);
\fill[mygray] (20.7mm,20.7mm) rectangle (29.3mm,29.3mm);

\draw (-0.6,1) node{\scriptsize\texttt{P(m)}};
\draw (-0.6,2) node{\scriptsize\texttt{V(m)}};
\draw (1,-0.6) node[rotate=-90]{\scriptsize\texttt{P(m)}};
\draw (2,-0.6) node[rotate=-90]{\scriptsize\texttt{V(m)}};
\begin{scope}[xshift=6.5cm]
\draw[densely dotted] (0,0) rectangle (3,3);
\fill(0,0) circle (0.04);
\fill(0,1.5) circle (0.04);
\fill(0,3) circle (0.04);
\fill(1.5,0) circle (0.04);
\fill(1.5,1.5) circle (0.04);
\fill(1.5,3) circle (0.04);
\fill(3,0) circle (0.04);
\fill(3,1.5) circle (0.04);
\fill(3,3) circle (0.04);
\draw (-0.6,1.5) node{\scriptsize\texttt{W(b)}};
\draw (1.5,-0.6) node[rotate=-90]{\scriptsize\texttt{W(b)}};
\draw[->,>=stealth,shorten >=0.7mm,shorten <=0.7mm] (0  ,0  ) -- (0  ,1.5);
\draw[->,>=stealth,shorten >=0.7mm,shorten <=0.7mm] (0  ,1.5) -- (0  ,3  );
\draw[->,>=stealth,shorten >=0.7mm,shorten <=0.7mm] (1.5,0  ) -- (1.5,1.5);
\draw[->,>=stealth,shorten >=0.7mm,shorten <=0.7mm] (1.5,1.5) -- (1.5,3  );
\draw[->,>=stealth,shorten >=0.7mm,shorten <=0.7mm] (3  ,0  ) -- (3  ,1.5);
\draw[->,>=stealth,shorten >=0.7mm,shorten <=0.7mm] (3  ,1.5) -- (3  ,3  );
\draw[->,>=stealth,shorten >=0.7mm,shorten <=0.7mm] (0  ,0  ) -- (1.5,0  );
\draw[->,>=stealth,shorten >=0.7mm,shorten <=0.7mm] (1.5,0  ) -- (3  ,0  );
\draw[->,>=stealth,shorten >=0.7mm,shorten <=0.7mm] (0  ,1.5) -- (1.5,1.5);
\draw[->,>=stealth,shorten >=0.7mm,shorten <=0.7mm] (1.5,1.5) -- (3  ,1.5);
\draw[->,>=stealth,shorten >=0.7mm,shorten <=0.7mm] (0  ,3  ) -- (1.5,3  );
\draw[->,>=stealth,shorten >=0.7mm,shorten <=0.7mm] (1.5,3  ) -- (3  ,3  );
\fill[mygray] (0.7mm,0.7mm) rectangle (14.3mm,14.3mm);
\fill[mygray] (15.7mm,15.7mm) rectangle (29.3mm,29.3mm);
\fill[mygray] (15.7mm,0.7mm) rectangle (29.3mm,14.3mm);
\fill[mygray] (0.7mm,15.7mm) rectangle (14.3mm,29.3mm);
\end{scope}
\end{tikzpicture}
\end{center}
The collection of \emph{higher dimensional automata} thus contains all automata (seen as the \(1\)-dimensional HDA) and provide a very expressive model of `true' concurrency~\cite{Pratt91,vanGlabbeek91,vanGlabbeek06}.  

\paragraph{From discrete to continuous.} 
Precubical sets with \emph{degeneracies}, \emph{a.k.a cubical sets} in algebraic topology, provide models for combinatorial homotopy theory, which is related to usual homotopy theory of spaces by the \emph{geometric realization} functor \(|\_|\)~\cite[p.~371-372]{BHS11}. 
However, in the setting of concurrency theory, spaces are not enough since they lack a notion of directedness. One of the many possible solutions is to add a notion of \emph{local order}: a \emph{locally ordered space} (\S\ref{sec:locally_ordered_spaces}) is a space equipped with an open base of ordered sets, and a \emph{directed path} is a continuous path which is locally increasing. We again get a \emph{locally ordered realization} functor \(|\_|_{lo}\) (\S\ref{sec:lorps}) satisfying \(|\_|=S\circ|\_|_{lo}\) with \(S\) the forgetful functor 
into topological spaces.
%
%
In addition, the tensor product is associative (up to isomorphism), and satisfies 
\(|K\otimes K'|_{lo}=|K|_{lo}\times|K'|_{lo}\) for every finite 
precubical sets \(K\) and \(K'\), so we define the \emph{continuous representation} 
of every concurrent program \(P_1|\cdots|P_n\) as the cartesian product 
\(|\interpret{P_1}|_{lo}\times\cdots\times|\interpret{P_n}|_{lo}\). 
In particular we have 
\begin{equation}\label{eqn:continuous_representation}
|(\cdot\to\cdots\to\cdot)^{\otimes^n}|_{lo}\cong([0,1],\leqslant)^n\hspace{3mm}\text{for 
every \(n\not=0\).} 
\end{equation}

\paragraph{Execution traces.} 
A graph made of finitely many consecutive arrows, \ie 
\(\{\cdot\to\cdots\to\cdot\}\), is said to be \emph{linear}; by extension, a process 
\(P\) is said to be \emph{linear} when so is the underlying graph of 
\(\interpret P\). 
Let \(G\) be a graph. 
A \emph{path} on \(G\) (\resp an automaton \(A\) over \(G\)) is a morphism \(\gamma:L\to G\) (\resp \(A\circ\gamma\)) with 
\(L\) linear.
The execution traces of a process \(P\) can be seen as paths on 
\(\interpret{P}\), hence the importance of this notion. However, it is 
unsufficient for concurrent programs because a path (as defined above) 
only visits elements of dimension \(0\) and \(1\), hence does not take into account parallel executions of instructions. 
The concept of \emph{paths on HDA} was introduced to address this issue 
\cite[Def.~17]{vanGlabbeek06}. The approach is purely combinatorial, 
but every path on an HDA \(A\) is obtained as the 
\emph{discretization} of some \emph{directed path} on \(|K|_{lo}\) 
(with \(K\) the underlying precubical set of \(A\)) as suggested by the 
following picture: 
\begin{center}
\begin{tikzpicture}[scale=0.8]
\foreach \x in {0, 1, 2} {
  \foreach \y in {0, 1, 2} {
    \fill (\x,\y) circle (0.03);
  }
}
\foreach \a in {0, 1, 2}
\foreach \x in {0, 1} {
  \draw[->,shorten <=2,shorten >=2,> = stealth] (\x,\a) -- (\x+1,\a);
  \draw[->,shorten <=2,shorten >=2,> = stealth] (\a,\x) -- (\a,\x+1);
}
\foreach \x in {0.1, 1.1} {
  \foreach \y in {0.1, 1.1} {
    \fill[opacity = 0.2] (\x,\y) -- (\x+0.8,\y) -- (\x+0.8,\y+0.8) -- (\x,\y+0.8) -- cycle ;
  }
}
\draw (0,1);
\draw (1,0);
\draw (1,1);
\draw (2.95,1) node {\(\mapsto\)};
\draw[very thick, opacity = 0.3, rounded corners] (0,0) -- (0.5,0.5) -- (1.5,0.5) -- (1.5,1.5) -- (2,2);
\begin{scope}[xshift=3.75cm]
\fill (0,0) circle (0.03);
\fill (2,2) circle (0.03);
\draw[->,shorten <=2,shorten >=2,> = stealth] (1,0) -- (1,1);
\draw[->,shorten <=2,shorten >=2,> = stealth] (1,1) -- (2,1);
\foreach \x / \y in {0.1 / 0.1 , 1.1 / 0.1 , 1.1 / 1.1} {
    \fill[opacity = 0.2] (\x,\y) -- (\x+0.8,\y) -- (\x+0.8,\y+0.8) -- (\x,\y+0.8) -- cycle ;
}
\end{scope}
\begin{scope}[xshift=9cm]
\foreach \x in {0, 1, 2} {
  \foreach \y in {0, 1, 2} {
    \fill (\x,\y) circle (0.03);
  }
}
\foreach \a in {0, 1, 2}
\foreach \x in {0, 1} {
  \draw[->,shorten <=2,shorten >=2,> = stealth] (\x,\a) -- (\x+1,\a);
  \draw[->,shorten <=2,shorten >=2,> = stealth] (\a,\x) -- (\a,\x+1);
}
\foreach \x in {0.1, 1.1} {
  \foreach \y in {0.1, 1.1} {
    \fill[opacity = 0.2] (\x,\y) -- (\x+0.8,\y) -- (\x+0.8,\y+0.8) -- (\x,\y+0.8) -- cycle ;
  }
}
\draw (0,1);
\draw (1,0);
\draw (1,1);
\draw (2.95,1) node {\(\mapsto\)};
\draw[very thick, opacity = 0.3, rounded corners] (0,0) --  (2,2);
\end{scope}
\begin{scope}[xshift=12.75cm]
\fill (0,0) circle (0.03);
\fill (1,1) circle (0.03);
\fill (2,2) circle (0.03);
\foreach \x / \y in {0.1 / 0.1 ,  1.1 / 1.1} {
    \fill[opacity = 0.2] (\x,\y) -- (\x+0.8,\y) -- (\x+0.8,\y+0.8) -- (\x,\y+0.8) -- cycle ;
}
\end{scope}
\end{tikzpicture}
\end{center}
\paragraph{Geometric models.} 
In the early years of geometry of concurrency, only parallel 
composition of linear processes were considered; 
models of such programs are finite unions of \(n\)-boxes, \ie set 
products \(I_1\times\cdots\times I_n\) with \(I_k\) nonempty intervals 
\cite{carson87geometry}. 
The combined use of topology and order in concurrency theory stems from 
this fact \cite{fajstrup2006algebraic}.
By carefully interpreting the synchronization instructions \texttt{P(\_)}, \texttt{V(\_)}, and \texttt{W(\_)}, one defines the \emph{geometric model} of any parallel composition \(P_1|\ldots|P_n\) of \emph{conservative} processes as a locally ordered subspace of \(|\interpret{P_1}|_{lo}\times\cdots\times|\interpret{P_n}|_{lo}\) \cite[Def.~4.1 and 6.2]{haucourt2018geo}, see Fig.~\ref{fig:geometric_models_of_toy_programs}. 
The metric of \(\mathbb{R}\) induces a metric on \(|G|_{lo}\) 
for every graph \(G\), and by extension 
on the geometric model of \(P_1|\ldots|P_n\). 
In such a model, two directed paths that are close according to the 
\emph{uniform distance} induce equivalent sequences of multi-instructions, \ie the state of the `ambient system' (virtual machine / hardware) is the same at the end of their executions, provided it was the same at the beginning \cite[Thm.~6.1]{haucourt2018geo}. 
\begin{figure}
  \begin{center}
    
\begin{tikzpicture} 
\begin{scope}
\def \eps {0.1};
\def \axs {-0.3};
\def \instshft {0.5};
  \foreach \x / \inst in {
    0/{}, 
    1/{\texttt{P(m)}}, 
    2/{\texttt{V(m)}}, 
    } {
    \draw (\x,\axs-\instshft) node[rotate=-90]{\scriptsize \inst};
    \fill (\x,\axs) circle (0.03);
  }
  \foreach \x / \inst in {
    0/{}, 
    1/{\texttt{P(m)}}, 
    2/{\texttt{V(m)}}, 
    } {
    \draw (\axs-\instshft,\x) node{\scriptsize\inst};
    \fill (\axs,\x) circle (0.03);
  }
\draw[->,>=stealth] (0,\axs) -- (3,\axs) ;
\draw[->,>=stealth] (\axs,0) -- (\axs,3) ;
{
\draw (-0.3,-0.3) node {$\times$}; 
\draw[densely dotted]  plot [smooth] coordinates 
{ 
  (0.2,0.1)  
  (0.6,1.8)  
  (1.8,2.7)  
  (2.7,2.9)  
};
\filldraw[fill opacity=0.15] (1,1) rectangle (2,2);
};
\draw (1.5,3.3) node[scale=0.8]{\(\R^2\setminus]0,1[^2\)};
\begin{scope}[xshift=6cm]
\def \eps {0.1};
\def \axs {-0.3};
\def \instshft {0.5};
    \draw (1.5,\axs-\instshft) node[rotate=-90]{\scriptsize\texttt{W(b)}};
    \draw (\axs-\instshft,1.5) node{\scriptsize\texttt{W(b)}};
    \fill (1.5,\axs) circle (0.03);
    \fill (\axs,1.5) circle (0.03);
\draw[->,>=stealth] (0,\axs) -- (3,\axs) ;
\draw[->,>=stealth] (\axs,0) -- (\axs,3) ;
\draw (0,1.5) -- (3,1.5) ;
\draw (1.5,0) -- (1.5,3) ;
{
\draw (-0.3,-0.3) node {$\times$}; 
\fill[white] (1.5,1.5) circle (0.03); 
\draw[densely dotted]  plot [smooth] coordinates 
{ 
  (0.2,0.1)  
  (0.7,0.5) 
  (1.5,1.5) 
  (2.5,2.1) 
  (2.8,2.6) 
};
};
\draw (1.2,3.3) node[scale=0.8]{\(\set{(x,y\in\R^2)|x=0\Leftrightarrow y=0}\)};
\end{scope}
\end{scope}
\begin{scope}[xshift=3mm,yshift=-25mm]
\draw[->,>=stealth] (0,0.65) node[left=1.75mm]{\scriptsize horizontal} -- (2.5,0.65) ;
\draw[->,>=stealth] (0,0) node[left=1.75mm]{\scriptsize vertical} -- (2.5,0) ; 
\fill (0.5,0) node[above]{\scriptsize\texttt{P(m)}} circle (0.025);
\fill (1,0) node[above]{\scriptsize\texttt{V(m)}} circle (0.025);
\fill (1.5,0.65) node[above]{\scriptsize\texttt{P(m)}} circle (0.025);
\fill (2,0.65) node[above]{\scriptsize\texttt{V(m)}} circle (0.025); 
\end{scope}
\begin{scope}[xshift=60mm,yshift=-25mm]
\draw[->,>=stealth] (0,0.65)  -- (2.5,0.65) ; 
\draw[->,>=stealth] (0,0)  -- (2.5,0) ;  
\fill (1.25,0.65) node[above]{\scriptsize\texttt{W(b)}} circle (0.025);
\fill (1.25,0) node[above]{\scriptsize\texttt{W(b)}} circle (0.025);
\end{scope}
\end{tikzpicture}
  \end{center}
\caption{Geometric models of toy programs (\ref{eqn:programs}).}
\label{fig:geometric_models_of_toy_programs}
\end{figure} 

\paragraph{Smooth models and graph blowups.} 
The \emph{length} of a \emph{smooth} path \(\gamma:[a,b]\to\R^n\) (relatively to a norm \(\norm\_\) on \(\R^n\)) is defined as \(\int_a^b\norm{\dot\gamma(t)}dt\) with \(\dot\gamma(t)\) the tangent vector to \(\gamma\) at \(t\in{]}a,b{[}\). 
We say that \(\gamma\) is \emph{directed} when the coordinates of \(\dot\gamma(t)\) in the standard basis of \(\R^n\) are non-negative. 
For every \(p\), \(q\in\R^n\) we have 
\[
\norm{p-q}\quad=\quad
\min\big\{\text{length}(\gamma)\:\big|\:\text{\(\gamma\) path from \(p\) to \(q\)}\big\}
\] 
and \(p\leqslant^nq\) iff there exists a directed path from \(p\) to \(q\). 
The geometric model of a parallel composition of linear processes is an ordered subspace \(X\) of \(\R^n\) for which the following hold \cite[p.~3]{haucourt2024non} (see Fig.~\ref{fig:continuous_vs_smooth}): 
\begin{itemize}
  \item[i)] For any directed path \(\gamma\) on \(X\), there is a directed \emph{smooth} path \(\delta\) on \(X\) that is arbitrarily close to (and with the same endpoints as) \(\gamma\), and whose length is at most that of \(\gamma\).
  \item[ii)] If \(\delta\) is close enough to \(\gamma\), then the execution traces they induce have the same effect on the state of the system. 
  \item[iii)] The length of a directed path (relatively to the \emph{uniform norm}) is the \emph{execution time} of the corresponding execution trace (assuming that concurrent execution is enabled). 
  \item[iv)] A smooth path on \(X\) is directed (\ie order-preserving in all coordinates) iff all its tangent vectors belong to \(\R_{\scalebox{0.8}{+}}^n\). 
\end{itemize}
\begin{figure}
\begin{center}
\begin{tikzpicture}
\draw (-0.5,0) -- (3,0) ;
\draw (0,-0.5) -- (0,3) ;
\foreach \p / \i in { 1 / {P(m)} , 2 / {V(m)}}
{
\draw (-0.1,\p) -- (0,\p) ;;
\draw (\p,-0.1) -- (\p,0) ;;
\draw (0,\p) node[left] {\texttt{\scriptsize{\i}}} ;
\draw (\p,0) node[rotate = 90 , left] {\texttt{\scriptsize{\i}}} ;  
}
\fill[opacity=0.2] (1,1) -- (1,2) -- (2,2) -- (2,1) -- cycle ;
\draw (1,2) -- (1,1) -- (2,1) ;
\draw[line width=0.05mm](0.000000,0.000000)--(0.030135,0.012386)--(0.109626,0.012848)--(0.212191,0.023422)--(0.220210,0.025576)--(0.268193,0.028708)--(0.277378,0.031362)--(0.297594,0.031538)--(0.461039,0.036067)--(0.536088,0.040763)--(0.598447,0.051220)--(0.620152,0.074004)--(0.630255,0.082713)--(0.638390,0.084896)--(0.738085,0.088879)--(0.764932,0.105618)--(0.792586,0.108560)--(0.794948,0.108583)--(0.797810,0.108681)--(0.806071,0.108870)--(0.818626,0.108908)--(0.822861,0.112245)--(0.836023,0.114688)--(0.839345,0.114955)--(0.840282,0.115003)--(0.848171,0.117051)--(0.857117,0.123443)--(0.858711,0.127774)--(0.859420,0.127913)--(0.915530,0.127971)--(0.950352,0.135941)--(0.950545,0.141558)--(0.950704,0.143776)--(0.951686,0.143783)--(0.952940,0.145934)--(0.953971,0.151513)--(0.955158,0.152904)--(0.962904,0.153259)--(0.970509,0.155171)--(0.974285,0.156227)--(0.975719,0.156237)--(0.978887,0.157468)--(0.979557,0.161448)--(0.979672,0.163774)--(0.980069,0.172063)--(0.981422,0.193472)--(0.982553,0.197873)--(0.991368,0.205105)--(0.991472,0.285099)--(0.991598,0.285781)--(0.994049,0.285791)--(0.994253,0.288155)--(0.995134,0.288897)--(0.996040,0.289148)--(0.998630,0.289148)--(0.999194,0.291415)--(0.999687,0.292854)--(0.999742,0.293980)--(0.999783,0.295742)--(0.999783,0.295852)--(0.999792,0.296049)--(0.999796,0.296071)--(0.999864,0.296239)--(0.999964,0.299928)--(1.000000,0.300000)--(1.003005,0.304511)--(1.036255,0.304523)--(1.151010,0.306558)--(1.158995,0.308745)--(1.171141,0.308976)--(1.172365,0.309110)--(1.176999,0.311156)--(1.183236,0.313951)--(1.184908,0.314335)--(1.185072,0.314874)--(1.186003,0.315114)--(1.186601,0.315566)--(1.192443,0.316110)--(1.193324,0.316454)--(1.200385,0.316514)--(1.201074,0.316515)--(1.203913,0.316554)--(1.206730,0.316581)--(1.207340,0.316629)--(1.211142,0.316837)--(1.211357,0.316848)--(1.211386,0.317312)--(1.211646,0.318392)--(1.211675,0.318402)--(1.212991,0.318420)--(1.233342,0.318461)--(1.234077,0.318652)--(1.235169,0.318748)--(1.361219,0.319214)--(1.467230,0.319292)--(1.479105,0.319364)--(1.484484,0.319403)--(1.484657,0.363218)--(1.484734,0.559697)--(1.484747,0.560423)--(1.484760,0.572258)--(1.485632,0.619741)--(1.486154,0.650411)--(1.487149,0.651999)--(1.490323,0.652321)--(1.491580,0.652348)--(1.491755,0.652354)--(1.493767,0.652363)--(1.493805,0.652376)--(1.493887,0.653535)--(1.493949,0.653549)--(1.494049,0.653780)--(1.494582,0.653804)--(1.494702,0.681870)--(1.494718,0.699271)--(1.496093,0.727455)--(1.496661,0.727782)--(1.496666,0.727802)--(1.496689,0.729092)--(1.496690,0.729968)--(1.496691,0.731920)--(1.496714,0.732421)--(1.496717,0.732597)--(1.497026,0.732924)--(1.498198,0.733028)--(1.499360,0.733187)--(1.499495,0.749151)--(1.499654,0.749630)--(1.500000,0.750000)--(1.510009,0.756941)--(1.527029,0.758767)--(1.572075,0.758906)--(1.606044,0.758939)--(1.607875,0.764019)--(1.618561,0.764153)--(1.618639,0.781445)--(1.632353,0.784760)--(1.639565,0.786287)--(1.656014,0.787092)--(1.658287,0.793510)--(1.658385,0.797487)--(1.658538,0.804462)--(1.665669,0.811257)--(1.675265,0.811340)--(1.758002,0.811607)--(1.791437,0.812372)--(1.794330,0.813624)--(1.799294,0.814572)--(1.808340,0.816133)--(1.808684,0.816511)--(1.811085,0.816583)--(1.825084,0.816613)--(1.827800,0.816669)--(1.830873,0.816892)--(1.832107,0.818073)--(1.930541,0.819392)--(1.990296,0.822833)--(1.995620,0.822991)--(2.008161,0.822993)--(2.020115,0.823006)--(2.020589,0.823066)--(2.025621,0.823342)--(2.027554,0.823439)--(2.028447,0.823552)--(2.029115,0.823568)--(2.032920,0.823611)--(2.037920,0.823748)--(2.038524,0.823928)--(2.039671,0.824213)--(2.039890,0.824612)--(2.039989,0.825791)--(2.040538,0.827278)--(2.040586,0.827910)--(2.041075,0.828037)--(2.044232,0.829033)--(2.045433,0.829783)--(2.046152,0.830707)--(2.046597,0.831641)--(2.049600,0.832312)--(2.050904,0.833267)--(2.053985,0.834564)--(2.054178,0.843194)--(2.056239,0.864711)--(2.057381,0.876730)--(2.060680,0.879059)--(2.080632,0.879133)--(2.085446,0.880069)--(2.085690,0.880289)--(2.086230,0.881481)--(2.091111,0.891785)--(2.094940,0.899865)--(2.096177,0.899988)--(2.100000,0.900000)--(2.101373,0.913170)--(2.107516,0.957947)--(2.108176,0.984401)--(2.114330,1.001743)--(2.114354,1.008966)--(2.114365,1.053422)--(2.114431,1.335999)--(2.114493,1.343691)--(2.116632,1.345329)--(2.121883,1.377420)--(2.122569,1.390986)--(2.122622,1.587011)--(2.131761,1.601114)--(2.145137,1.601834)--(2.145493,1.641023)--(2.147227,1.644197)--(2.148014,1.644389)--(2.148754,1.644478)--(2.149056,1.645044)--(2.149301,1.646013)--(2.149346,1.656937)--(2.149354,1.669098)--(2.149358,1.669428)--(2.149359,1.671132)--(2.149675,1.672749)--(2.150037,1.684594)--(2.157742,1.702824)--(2.162284,1.723729)--(2.162336,1.725782)--(2.163828,1.728892)--(2.165141,1.732596)--(2.173739,1.758906)--(2.173802,1.782820)--(2.174042,1.784297)--(2.174782,1.784306)--(2.175047,1.793934)--(2.175826,1.859999)--(2.177367,1.868786)--(2.177688,1.914386)--(2.177755,1.946210)--(2.177934,1.955885)--(2.178172,1.961841)--(2.178600,1.963068)--(2.178670,1.966473)--(2.179192,1.980149)--(2.179319,1.993773)--(2.180140,1.993814)--(2.180441,1.993911)--(2.180715,1.993974)--(2.180820,1.994051)--(2.181333,1.994092)--(2.181575,1.994136)--(2.184831,1.994181)--(2.186893,1.994242)--(2.187139,1.994244)--(2.188219,1.994373)--(2.199722,1.994776)--(2.223174,1.994803)--(2.223474,1.996499)--(2.227484,1.997567)--(2.259475,1.998340)--(2.273441,1.998601)--(2.285131,1.999985)--(2.300000,2.000000)--(2.303182,2.000002)--(2.303793,2.000008)--(2.306635,2.000169)--(2.311758,2.000262)--(2.317087,2.000276)--(2.319445,2.000288)--(2.348404,2.000295)--(2.363191,2.000302)--(2.391446,2.002895)--(2.399815,2.003375)--(2.400311,2.003437)--(2.413003,2.003510)--(2.414348,2.008822)--(2.414357,2.022130)--(2.415524,2.026552)--(2.417862,2.030595)--(2.419508,2.054201)--(2.420222,2.153631)--(2.421265,2.217580)--(2.421905,2.238207)--(2.421910,2.239765)--(2.421930,2.240149)--(2.421951,2.330399)--(2.421996,2.353739)--(2.422018,2.354364)--(2.422042,2.354626)--(2.422091,2.358748)--(2.422159,2.362482)--(2.422165,2.362869)--(2.422165,2.363312)--(2.422186,2.363522)--(2.422192,2.364778)--(2.425003,2.364925)--(2.432755,2.364971)--(2.435160,2.364981)--(2.435686,2.365071)--(2.435888,2.372054)--(2.437739,2.372701)--(2.438116,2.374371)--(2.438496,2.375863)--(2.453577,2.380977)--(2.462185,2.381722)--(2.462344,2.383409)--(2.462708,2.383491)--(2.474588,2.384092)--(2.713626,2.384631)--(2.739767,2.384633)--(2.847009,2.384654)--(2.849169,2.386274)--(2.849327,2.387471)--(2.851378,2.387473)--(2.854483,2.387644)--(2.854829,2.388754)--(2.858156,2.389770)--(2.858490,2.392024)--(2.858957,2.397815)--(2.876009,2.397861)--(2.885009,2.397908)--(2.896541,2.398279)--(2.897212,2.398695)--(2.898493,2.398703)--(2.899779,2.398717)--(2.899787,2.399579)--(2.900000,2.400000)--(2.905037,2.418772)--(2.906239,2.418856)--(2.907359,2.424282)--(2.908510,2.427580)--(2.909246,2.432595)--(2.909314,2.433174)--(2.913175,2.433526)--(2.917962,2.437964)--(2.931365,2.438405)--(2.931941,2.440571)--(2.932174,2.442791)--(2.934161,2.444755)--(2.935035,2.448159)--(2.935469,2.450349)--(2.946057,2.460721)--(2.950165,2.464085)--(2.950191,2.464755)--(2.950263,2.466870)--(2.950264,2.468074)--(2.950276,2.468335)--(2.950291,2.471369)--(2.950403,2.474259)--(2.950446,2.477254)--(2.950449,2.478997)--(2.950462,2.481865)--(2.950500,2.482086)--(2.952369,2.507595)--(2.952855,2.514144)--(2.952904,2.515615)--(2.952922,2.526630)--(2.952926,2.527665)--(2.952936,2.528501)--(2.952956,2.560644)--(2.952980,2.577877)--(2.953244,2.581998)--(2.953283,2.582121)--(2.953290,2.596625)--(2.953302,2.598992)--(2.953339,2.599672)--(2.953699,2.600084)--(2.972188,2.601346)--(2.979716,2.608338)--(2.983813,2.610982)--(2.991666,2.613563)--(2.992049,2.614821)--(2.997049,2.617908)--(2.997360,2.618358)--(2.997870,2.620464)--(2.997972,2.622364)--(2.998020,2.623045)--(2.998101,2.625856)--(2.998119,2.628279)--(2.998155,2.641277)--(2.998313,2.888169)--(2.998324,2.890298)--(2.998333,2.891697)--(2.998372,2.895713)--(2.998649,2.899091)--(2.999261,2.912912)--(2.999401,2.920177)--(2.999638,2.928667)--(2.999919,2.937484)--(2.999948,2.975598)--(3.000000,3.000000);
\fill (1.0,0.3) circle (0.027);
\fill (2.0,0.825) circle (0.027);
\fill (2.115,1.0) circle (0.027);
\fill (2.3,2.0) circle (0.027);
\draw[line width=0.05mm] plot [smooth] coordinates {(0.0,0.0) (1.0,0.3) (1.64,0.5) (2.115,1.0) (2.3,2.0) (2.9,2.4) (3.0,3.0)};
\draw (1.35,0.7) node{\(\gamma\)};
\draw (1.93,0.5) node{\(\delta\)};
\begin{scope}[xshift=60mm,yshift=10mm]
\draw[->,>=stealth] (0,0.65) node[left]{\scriptsize horizontal} -- (2.5,0.65) ;
\draw[->,>=stealth] (0,0) node[left]{\scriptsize vertical} -- (2.5,0) ; 
\fill (1.5,0) node[above]{\scriptsize\texttt{P(m)}} circle (0.025);
\fill (2,0) node[above]{\scriptsize\texttt{V(m)}} circle (0.025);
\fill (0.5,0.65) node[above]{\scriptsize\texttt{P(m)}} circle (0.025);
\fill (1,0.65) node[above]{\scriptsize\texttt{V(m)}} circle (0.025); 
\end{scope}
\end{tikzpicture}
\end{center}
\caption{Continuous vs smooth directed paths.}
\label{fig:continuous_vs_smooth}
\end{figure}

If one of the processes \(P_1,\ldots,P_n\) is not linear, then one of 
the (graphs of the) automata \(\interpret{P_1},\ldots,\interpret{P_n}\) 
has a vertex with at least two outgoing (\resp ingoing) arrows, so the 
locally ordered space 
\(|\interpret{P_1}|_{lo}\times\cdots\times|\interpret{P_1}|_{lo}\) has 
a point whose neighborhoods are not isomorphic with 
\((\R^n,\leqslant^n)\). 
This difficulty can be overcome by means of \emph{blowups}:
\begin{theorem}[{\cite[\S 5]{haucourt2024non}}]\label{thm:blowup_graphs}
    Let $X$ be a product of the locally ordered realizations of $n$ graphs $G_i$, $i\in\{1,\dots,n\}$, i.e. $X=\prod_i|G_i|\cong |\bigotimes_i G_i|$. There is a euclidean local order $\Tilde{X}$, called the blowup of $X$, and a local embedding of locally ordered spaces $\beta_X:\Tilde{X}\xrightarrow{}X$, called the blowup map, satisfying the following universal property: for any euclidean local order $E$ of dimension $n$, for any local embedding $f:E\xrightarrow{}X$, there is a unique continuous map $\Tilde{f}:E\xrightarrow{}\Tilde{X}$ such that \(f=\beta_X\circ\tilde f\).
    Moreover, $\Tilde{f}$ is a local embedding.
\end{theorem}
Then assertions (i)-(iv) hold with \(\delta=\beta_X\circ\tilde\delta\) 
and \(\tilde\delta\) a well-chosen smooth path on \(\beta_X^{-1}(M)\) 
the inverse image of the geometric model \(M\) under blowup map 
\(\beta_X\). 
An inescapable `drawback' is that \(\tilde X\) is \emph{non-Hausdorff} 
(\(\tilde X\) is an open subspace of \(X\) precisely when all the 
processes \(P_1,\ldots,P_n\) are linear).

\paragraph{Traversals.} In this article, we generalize Thm.~\ref{thm:blowup_graphs} to every locally ordered space \(X\). The key observation that allows us to do so is that the points of the \(n\)-dimensional blowup of \(X\) (\(n\in\N\)) are its \emph{\(n\)-traversals} which, as we will see, correspond to the germs of a sheaf on $X$:

\begin{definition}\label{def:traversal}
Two subsets \(A\) and \(B\) of a locally ordered space \(X\) are said to be \emph{equivalent} at \(p\), which we denote by \(A\sim_pB\) when there exists a neighborhood \(V\) of \(p\) such that \(A\cap V=B\cap V\); the relation \(\sim_p\) is an equivalence. 
The \emph{germ} at \(p\) of \(A\subseteq X\), which we denote by \(A_p\), is the \(\sim_p\)-equivalence class of \(A\).  
The \(n\)-\emph{traversals} at \(p\) (for \(n\in\mathbb N\)) are the germs at \(p\) of the subsets \(A\subseteq X\) such that \(p\in A\), $A$ is a subset of some ordered open set $U$ of $X$, and \(A\cong\R^n\) as ordered spaces. 
\end{definition}

\begin{figure}[h]
\begin{center}
\begin{tikzpicture}
\draw (0:0.2) node {\scalebox{0.75}{\(v\)}} ;  
\draw (-0.75,0.2) node {\scalebox{0.75}{\(a\)}} ;  
\draw (80:0.75) node {\scalebox{0.75}{\(b\)}} ;  
\draw (-80:0.75) node {\scalebox{0.75}{\(c\)}} ;  
\draw (135:0.8) node[yshift=-10mm] {\scalebox{0.75}{\(G\)}} ;  
\filldraw[fill=white] (0,0) circle (0.025) ;
\filldraw[fill=white] (60:1.5) circle (0.025) ;
\filldraw[fill=white] (180:1.5) circle (0.025) ;
\filldraw[fill=white] (300:1.5) circle (0.025) ; 
\draw[->,>=stealth,shorten >=2,shorten <=2] (180:1.5) -- (0,0) ;
\draw[->,>=stealth,shorten >=2,shorten <=2] (0,0) -- (60:1.5) ; 
\draw[->,>=stealth,shorten >=2,shorten <=2] (0,0) -- (300:1.5) ; 
\begin{scope}[xshift=28mm]
\draw (0:0.2) node[xshift=-3mm,yshift=-1.6mm] {\scalebox{0.75}{\(\set v\)}} ;  
\draw (-0.71,0.2) node {\scalebox{0.75}{\(\set{a}\times\coint0{1}\)}} ;  
\draw (80:0.7) node[xshift=5mm,rotate=60] {\scalebox{0.75}{\(\set{b}\times\ocint0{1}\)}} ;  
\draw (-80:0.7) node[xshift=5mm,rotate=-60] {\scalebox{0.75}{\(\set{c}\times\ocint0{1}\)}} ;  
\draw (135:0.8) node[yshift=-10mm] {\scalebox{0.75}{\(|G|\)}} ;  
\draw[arrows = {Arc Barb[width=3pt,arc=120]-}] (180:1.5) -- (0,0) ;
\draw[arrows = {-Arc Barb[width=3pt,arc=120]}] (0,0) -- (60:1.5) ;
\draw[arrows = {-Arc Barb[width=3pt,arc=120]}] (0,0) -- (300:1.5) ;
\end{scope} 
\begin{scope}[xshift=66mm]
\draw (0:0.2) node[xshift=3.6mm,yshift=3mm] {\scalebox{0.75}{\(\set {(a,b)}\)}} ;  
\draw (0:0.2) node[xshift=3.6mm,yshift=-3mm] {\scalebox{0.75}{\(\set {(a,c)}\)}} ;  
\draw (-0.8,0.25) node {\scalebox{0.75}{\(\set{a}\times\opint0{1}\)}} ;  
\draw (80:0.35) node[xshift=6mm,yshift=6mm] {\scalebox{0.75}{\(\set{b}\times\opint0{1}\)}} ;  
\draw (-80:0.35) node[xshift=6mm,yshift=-6mm] {\scalebox{0.75}{\(\set{c}\times\opint0{1}\)}} ;  
\draw (135:0.8) node[xshift=-3mm,yshift=-11mm] {\scalebox{0.75}{\(\text{blowup of \(|G|\)}\)}} ;  
\draw[arrows = {Arc Barb[reversed,width=3pt,arc=120]-Arc Barb[reversed,width=3pt,arc=120]}] (180:1.5) -- (-0.05,0) ;
\draw[arrows = {Arc Barb[reversed,width=3pt,arc=120]-Arc Barb[reversed,width=3pt,arc=120]}] (0.05,0.75) -- (1.5,0.75) ;
\draw[arrows = {Arc Barb[reversed,width=3pt,arc=120]-Arc Barb[reversed,width=3pt,arc=120]}] (0.05,-0.75) -- (1.5,-0.75) ;
\fill (0,0.375) circle (0.025) ;
\fill (0,-0.375) circle (0.025) ;
\end{scope}
\end{tikzpicture}
\end{center}
\caption{The blowup of a graph containing a branching point.}
\label{fig:blow_up_a_branching_point}
\end{figure}

Let \(G\) be a graph. 
The locally ordered space \(|G|_{\scalebox{0.7}{\emph{lo}}}\) is \(1\)-dimensional (\ie the greatest \(n\in\mathbb N\) for which \(|G|_{\scalebox{0.7}{\emph{lo}}}\) has a subspace isomorphic to \(\R^n\) is \(1\)). 
A point \(p\in|G|_{\scalebox{0.7}{\emph{lo}}}\) is \emph{singular} (or \emph{a singularity}) precisely when there are at least two \(1\)-traversals at \(p\). 
Every point of \(|G|_{lo}\) that is not a vertex of \(G\) has a unique \(1\)-traversal. 
The \(1\)-traversals at a vertex \(v\) can be described in purely combinatorial terms by means of arrows to and from \(v\), which is the approach followed by~\cite{haucourt2024non}. 
More precisely, the set of vertices of \(G\) form a discrete subspace of \(|G|_{lo}\), and \(|G|_{lo}\setminus\{\text{vertices of \(G\)}\}\) is an open subset of the blowup of \(|G|_{lo}\) whose complement (in the blowup) is discrete.
In other words, the blowup of $G$ is identical to \(|G|_{lo}\) over the edges of $G$. On the contrary, the points of $|G|_{lo}$ over a vertex $v$ of $G$ correspond to the pairs $(a,b)$ of edges of $G$ such that $v$ is a target of $a$ and a source of $b$, see Fig.~\ref{fig:blow_up_a_branching_point}.


%
%
%
%
%


We examine the prototypical case of the graph \(G\) from 
Fig.~\ref{fig:G}. 
The realization \(|G|_{lo}\) is isomorphic to the locally ordered subspace \(\{(x,y)\in[-1,1]^2;\:xy=0\}\) of \((\mathbb R,\leqslant)^2\).
We observe that \(|G|_{lo}\) is locally isomorphic to \((\mathbb R,\leqslant)\) (as a locally ordered space) at every point except \((-1,0)\), \((1,0)\), \((0,-1)\), \((0,1)\) and the origin. 
The latter is, 
as a \emph{branching point}, the only semantically meaningful singularity, the other four can be ignored~\cite[p.~10]{haucourt2024non}. 
\begin{figure}[ht!]
\begin{center}
\begin{tikzpicture}[scale=0.8]
\draw (-1,0) circle (0.025);
\draw (1,0) circle (0.025);
\draw (0,-1) circle (0.025);
\draw (0,1) circle (0.025);
\draw (0,0) circle (0.025);
\draw[->,>=stealth,shorten >=1.5mm,shorten <=2mm,] (0,0) -- (1,0) ;
\draw[->,>=stealth,shorten >=1.5mm,shorten <=2mm,] (-1,0) -- (0,0) ;
\draw[->,>=stealth,shorten >=1.5mm,shorten <=2mm,] (0,0) -- (0,1) ;
\draw[->,>=stealth,shorten >=1.5mm,shorten <=2mm,] (0,-1) -- (0,0) ;
\draw (0,-3/2) node {\(G\)} ;
\begin{scope}[xshift=25mm]
\draw[thick] (-1,0) -- (1,0) ;
\draw[thick] (0,-1) -- (0,1) ;
\draw (0,-3/2) node {\hspace{2.5mm}\(|G|_{\scalebox{0.7}{\emph{lo}}}\cong\{xy=0\}\)} ;
\end{scope}
\begin{scope}[xshift=50mm]
\draw[thick] (-1,0) -- (1,0) ;
\draw[densely dotted] (0,-1) -- (0,1) ;
\end{scope}
\begin{scope}[xshift=75mm]
\draw[densely dotted] (-1,0) -- (1,0) ;
\draw[thick] (0,-1) -- (0,1) ;
\end{scope}
\begin{scope}[xshift=100mm]
\draw[densely dotted] (-1,0) -- (1,0) ;
\draw[densely dotted] (0,-1) -- (0,1) ;
\draw[thick] (-1,0) -- (0,0) -- (0,1) ;
\end{scope}
\begin{scope}[xshift=125mm]
\draw[densely dotted] (-1,0) -- (1,0) ;
\draw[densely dotted] (0,-1) -- (0,1) ;
\draw[thick] (0,-1) -- (0,0) -- (1,0) ;
\end{scope}
\end{tikzpicture}
\end{center}
\caption{The \(4\) traversals at the origin in \(\{xy=0\}\).}\label{fig:G}
\end{figure}
The \(1\)-dimensional traversals at the origin of \(|G|_{\scalebox{0.7}{\emph{lo}}}\) (there are four of them) are shown on Fig.~\ref{fig:G};  
they correspond to the paths from the left and down vertices of \(G\) to its up and right ones. So in the blowup of $|G|_{lo}$, we have exactly four points over the origin.



\paragraph{Beyond graphs.} The main purpose of this article is to show that the class of precubical sets that can be represented by (non-Hausdorff) manifolds is much larger than that of tensor products of graphs. 
A practical motivation for such a theoretical development is to model programs during the execution of which the number of processes may vary.  
Let us give an example. 
Denote by \(I\) the graph \((-1)\to 0 \to 1\). 
The vertices (\resp arrows) of \(I\) are \(-1\), \(0\), and \(1\) (\resp \((-1)\hspace{-0.7mm}\to\hspace{-0.7mm}0\) and \(0\hspace{-0.7mm}\to\hspace{-1mm}1\)); their dimension is \(0\) (resp. \(1\)). 
The elements of \(C=I\otimes I\otimes I\) are the \(3\)-tuples \((x,y,z)\) of elements of \(I\) with \(\dim_C(x,y,z)=\dim_I(x)+\dim_I(y)+\dim_I(z)\). 
One of the simplest precubical set that is not a tensor product of graphs is the precubical subset \(K=\{(x,y,z)\in C;\:0\in\{x,y,z\}\}\) of \(C\).   
Its realization \(|K|_{lo}\) is the locally ordered subspace \(\{(x,y,z)\in[-1,1]^3;\:xyz=0\}\) of \((\R^3,\leqslant^3)\) (Fig.~\ref{fig:K}); it is \(2\)-dimensional.
Let \(p=(x,y,z)\in\:|K|_{lo}\), \(A(p)=xy+yz+zx\), and \(B(p)=(1-x^2)(1-y^2)(1-z^2)\). 
Note that \(B(p)=1\) iff \(p\) is the origin, and that \(A(p)=0\) iff \(p\) belongs to an axis.  
The subspace \(\set{\text{\(B=0\)}}\) is called the \emph{boundary} of \(|G|_{lo}\).
If \(A(p)B(p)\not=0\) then \(p\) has a neighborhood that is isomorphic to \((\R^2,\leqslant^2)\). 
If \(A(p)=0\) and \(B(p)\not\in\set{0,1}\), then \(p\) has a neighborhood that is isomorphic to \(\mathbb R\times\{(x,y)\in\mathbb R^2;\:xy=0\}\) (Fig.~\ref{fig:G}). 
The origin is the only point to which the method from~\cite{haucourt2024non} does not apply, whereas our blowup construction can handle it: 

The points of the blowup of \(K\) (a short for `the \(2\)-dimensional blowup of the locally ordered space \(|K|_{lo}\)') over $p$ are the \(2\)-traversals of \(|K|_{lo}\) at $p$; the number of traversals at a given \(p\) is 
\begin{itemize}
\item [0] if \(B(p)=0\),
\item [1] if \(A(p)B(p)\not=0\),
\item [4] if \(A(p)=0\) and \(B(p)\not\in\{0,1\}\), 
\item [9] if \(B(p)=1\) (\ie \(p\) is the origin), see Fig.~\ref{fig:K}. 
\end{itemize}
\vspace{0.2cm}
An exhaustive description of an \emph{atlas} of the blowup would be fastidious, nevertheless, by way of an example, a \emph{chart} \(\phi\) at the traversal \(\{yz=0\wedge y\geqslant0\wedge z\leqslant 0\}\) of the origin is given by \(\phi(x,0,z)=(x,z)\) and  \(\phi(x,y,0)=(x,y)\) (Fig.~\ref{fig:K}). 
In a similar fashion we obtain charts at the other traversals of the origin; it is easy to check that the \emph{transition maps} between such charts are identities (in particular they are smooth and order preserving).

\begin{figure}[ht!]
\begin{center}
\scalebox{0.7}{\begin{tikzpicture}
\draw[thin,->] (-5/3,-4/3) -- (-7/6,-4/3) node{\hspace{4mm}\scalebox{0.7}{\(x\)}} ; 
\draw[thin,->] (-5/3,-4/3) -- (-5/3,-5/6) node[above]{\scalebox{0.7}{\(y\)}} ; 
\draw[thin,->] (-5/3,-4/3) -- (-4/3,-7/6) ;
\draw (-14/12,-13/12) node {\scalebox{0.7}{\(z\)}};
\draw[thick] (0,0) -- (1,0) ;
\draw[thick] (0,0) -- (0,1) ;
\draw[thick] (1,0) -- (1,1) ;
\draw[thick] (0,1) -- (1,1) ;
\draw[thick] (0,-1) -- (1,-1) ;
\draw[thick] (1,-1) -- (1,0) ;
\draw[thick] (1,0) -- (5/3,1/3) ;
\draw[thick] (5/3,1/3) -- (1,1/3) ;
\draw[densely dotted] (1,1/3) -- (2/3,1/3) ;
\draw[densely dotted] (2/3,1/3) -- (0,0) ;
\draw[densely dotted] (2/3,1/3) -- (2/3,1) ;
\draw[thick] (2/3,1) -- (2/3,4/3) ;
\draw[thick] (0,0) -- (-2/3,-1/3) ;
\draw[thick] (-2/3,-1/3) -- (1/3,-1/3) ;
\draw[thick] (1/3,-1/3) -- (1,0) ;
\draw[thick] (0,1) -- (-1,1) ; 
\draw[thick] (-1,1) -- (-1,0) ; 
\draw[thick] (-5/3,-1/3) -- (-2/3,-1/3) ;
\draw[thick] (-5/3,-1/3) -- (-1,0) ;
\draw[thick] (0,1) -- (-2/3,2/3) ;
\draw[thick] (-2/3,2/3) -- (-2/3,-1/3) ;
\draw[thick] (0,1) -- (2/3,4/3) ;
\draw[thick] (-2/3,-1/3) -- (-2/3,-4/3) ;
\draw[thick] (-2/3,-4/3) -- (0,-1) ;
\draw[thick] (0,-1) -- (0,-1/3) ;
\draw[thick] (-1,0) -- (-2/3,0) ;
\draw[densely dotted] (-2/3,0) -- (0,0) ;
\draw[densely dotted] (0,-1/3) -- (0,0) ;
\draw[thick] (-1,-1) -- (-2/3,-1) ;
\draw[thick] (-1,-1) -- (-1,-1/3) ;
\draw[densely dotted] (-1,-1/3) -- (-1,0) ;
\draw[densely dotted] (0,-1) -- (2/3,-2/3) ;
\draw[densely dotted] (2/3,-2/3) -- (2/3,1/3) ;
\draw[densely dotted] (-1,0) -- (-1/3,1/3) ;
\draw[densely dotted] (-1/3,1/3) -- (2/3,1/3) ;
\draw[densely dotted] (-2/3,-1) -- (0,-1) ;
\begin{scope}[xshift=35mm]
\fill[mygray] (-1,-1) rectangle (1,1) ;
\draw[densely dotted] (-1,0) -- (1,0) ;
\draw[densely dotted] (-5/3,-1/3) -- (1/3,-1/3) ;
\draw[densely dotted] (1,1/3) -- (5/3,1/3) ;
\draw[densely dotted] (-5/3,-1/3) -- (-1,0) ;
\draw[densely dotted] (-2/3,-1/3) -- (0,0) ;
\draw[densely dotted] (1/3,-1/3) -- (5/3,1/3) ;
\draw[densely dotted] (0,-1) -- (0,1) ;
\draw[densely dotted] (-2/3,-4/3) -- (-2/3,2/3) ;
\draw[densely dotted] (2/3,1) -- (2/3,4/3) ;
\draw[densely dotted] (-2/3,2/3) -- (2/3,4/3) ;
\draw[densely dotted] (-2/3,-4/3) -- (0,-1) ;
\draw[thick] (-1,-1) rectangle (1,1) ;
\draw (0,-1.6) node {\scalebox{0.75}{\(\{z=0\}\)}} ;
\end{scope}
\begin{scope}[xshift=70mm]
\fill[mygray] (-5/3,-1/3) -- (1/3,-1/3) -- (5/3,1/3) -- (-1/3,1/3) -- cycle ;
\draw[densely dotted] (-1,0) -- (1,0) ;
\draw[densely dotted] (-2/3,-1/3) -- (2/3,1/3) ;
\draw[densely dotted] (0,-1) -- (0,-1/3) ;
\draw[densely dotted] (0,0) -- (0,1) ;
\draw[densely dotted] (1,-1) -- (1,1) ;
\draw[densely dotted] (-1,-1) -- (-1,-1/3) ;
\draw[densely dotted] (-1,0) -- (-1,1) ;
\draw[densely dotted] (-2/3,-4/3) -- (-2/3,2/3) ;
\draw[densely dotted] (2/3,-2/3) -- (2/3,-1/6) ;
\draw[densely dotted] (2/3,1/3) -- (2/3,4/3) ;
\draw[densely dotted] (-1,1) -- (1,1) ;
\draw[densely dotted] (-1,-1) -- (1,-1) ;
\draw[densely dotted] (-2/3,2/3) -- (2/3,4/3) ;
\draw[densely dotted] (-2/3,-4/3) -- (2/3,-2/3) ;
\draw[thick] (-5/3,-1/3) -- (1/3,-1/3) -- (5/3,1/3) -- (-1/3,1/3) -- cycle ;
\draw (0,-1.6) node {\scalebox{0.75}{\(\{y=0\}\)}} ;
\end{scope}
\begin{scope}[xshift=105mm]
\fill[mygray] (-2/3,-4/3) -- (2/3,-2/3) -- (2/3,4/3) -- (-2/3,2/3) -- cycle ;
\draw[densely dotted] (-1,0) -- (-2/3,0) ;
\draw[densely dotted] (0,0) -- (1,0) ;
\draw[densely dotted] (-5/3,-1/3) -- (1/3,-1/3) ;
\draw[densely dotted] (2/3,1/3) -- (5/3,1/3) ;
\draw[densely dotted] (-5/3,-1/3) -- (-2/3,1/6) ;
\draw[densely dotted] (-2/3,-1/3) -- (2/3,1/3) ;
\draw[densely dotted] (1/3,-1/3) -- (5/3,1/3) ;
\draw[densely dotted] (0,-1) -- (0,1) ;
\draw[densely dotted] (1,-1) -- (1,1) ;
\draw[densely dotted] (-1,-1) -- (-1,1) ;
\draw[densely dotted] (-1,1) -- (1,1) ;
\draw[densely dotted] (-1,-1) -- (-2/3,-1) ;
\draw[densely dotted] (0,-1) -- (1,-1) ;
\draw[thick] (-2/3,-4/3) -- (2/3,-2/3) -- (2/3,4/3) -- (-2/3,2/3) -- cycle ;
\draw (0,-1.6) node {\scalebox{0.75}{\(\{x=0\}\)}} ;
\end{scope}
\begin{scope}[xshift=140mm]
\fill[mygray] (-5/3,-1/3) -- (1/3,-1/3) -- (1,0) -- (1,1) -- (-1,1) -- (-1,0) -- cycle ;
\draw[thick] (-1,0) -- (1,0) ;
\draw[thick] (-5/3,-1/3) -- (1/3,-1/3) ;
\draw[densely dotted] (1,1/3) -- (5/3,1/3) ;
\draw[densely dotted] (-2/3,-1/3) -- (0,0) ;
\draw[densely dotted] (1/3,-1/3) -- (5/3,1/3) ;
\draw[densely dotted] (0,-1) -- (0,-1/3) ;
\draw[densely dotted] (0,0) -- (0,1) ;
\draw[densely dotted] (1,-1) -- (1,1) ;
\draw[densely dotted] (-1,-1) -- (-1,-1/3) ;
\draw[densely dotted] (-2/3,-4/3) -- (-2/3,2/3) ;
\draw[densely dotted] (2/3,-2/3) -- (2/3,-1/6) ;
\draw[densely dotted] (2/3,1) -- (2/3,4/3) ;
\draw[thick] (-1,1) -- (1,1) ;
\draw[densely dotted] (-1,-1) -- (1,-1) ;
\draw[densely dotted] (-2/3,2/3) -- (2/3,4/3) ;
\draw[densely dotted] (-2/3,-4/3) -- (2/3,-2/3) ;
\draw[thick] (-5/3,-1/3) -- (-1,0) -- (-1,1);
\draw[thick] (1/3,-1/3) -- (1,0) -- (1,1);
\draw (0,-1.6) node {\scalebox{0.75}{\(\{yz=0\wedge y\geqslant0\wedge z\leqslant0\}\)}} ;
\end{scope}
\begin{scope}[yshift=-35mm]
\fill[mygray] (-1,-1) -- (1,-1) -- (1,0) -- (5/3,1/3) -- (-1/3,1/3) -- (-1,0) -- cycle ;
\draw[thick] (-1,0) -- (1,0) ;
\draw[densely dotted] (-5/3,-1/3) -- (1/3,-1/3) ;
\draw[thick] (-1/3,1/3) -- (5/3,1/3) ;
\draw[densely dotted] (-5/3,-1/3) -- (-1/3,1/3) ;
\draw[densely dotted] (-2/3,-1/3) -- (2/3,1/3) ;
\draw[densely dotted] (1/3,-1/3) -- (5/3,1/3) ;
\draw[densely dotted] (0,-1) -- (0,1) ;
\draw[densely dotted] (1,-1) -- (1,1) ;
\draw[densely dotted] (-1,-1) -- (-1,1) ;
\draw[densely dotted] (-2/3,-4/3) -- (-2/3,2/3) ;
\draw[densely dotted] (2/3,1/3) -- (2/3,4/3) ;
\draw[densely dotted] (-1,1) -- (1,1) ;
\draw[thick] (-1,-1) -- (1,-1) ;
\draw[densely dotted] (-2/3,2/3) -- (2/3,4/3) ;
\draw[densely dotted] (-2/3,-4/3) -- (0,-1) ;
\draw[thick] (1,-1) -- (1,0) -- (5/3,1/3) ;
\draw[thick] (-1,-1) -- (-1,0) -- (-1/3,1/3) ;
\draw (0,-1.6) node {\scalebox{0.75}{\(\{yz=0\wedge y\leqslant0\wedge z\geqslant0\}\)}} ;
\end{scope}
\begin{scope}[xshift=35mm,yshift=-35mm]
\fill[mygray] (-1,-1) -- (0,-1) -- (2/3,-2/3) -- (2/3,4/3) -- (0,1) -- (-1,1) -- cycle ;
\draw[densely dotted] (-1,0) -- (1,0) ;
\draw[densely dotted] (-5/3,-1/3) -- (1/3,-1/3) ;
\draw[densely dotted] (2/3,1/3) -- (5/3,1/3) ;
\draw[densely dotted] (-5/3,-1/3) -- (-1,0) ;
\draw[densely dotted] (-2/3,-1/3) -- (2/3,1/3) ;
\draw[densely dotted] (1/3,-1/3) -- (5/3,1/3) ;
\draw[thick] (0,-1) -- (0,1) ;
\draw[densely dotted] (1,-1) -- (1,1) ;
\draw[thick] (-1,-1) -- (-1,1) ;
\draw[densely dotted] (-2/3,-4/3) -- (-2/3,2/3) ;
\draw[thick] (2/3,-2/3) -- (2/3,4/3) ;
\draw[densely dotted] (-1,1) -- (1,1) ;
\draw[densely dotted] (-1,-1) -- (1,-1) ;
\draw[densely dotted] (-2/3,2/3) -- (2/3,4/3) ;
\draw[densely dotted] (-2/3,-4/3) -- (2/3,-2/3) ;
\draw[thick] (-1,-1) -- (0,-1) -- (2/3,-2/3) ;
\draw[thick] (-1,1) -- (0,1) -- (2/3,4/3) ;
\draw (0,-1.6) node {\scalebox{0.75}{\(\{xz=0\wedge x\leqslant0\wedge z\geqslant0\}\)}} ;
\end{scope}
\begin{scope}[xshift=70mm,yshift=-35mm]
\fill[mygray] (-2/3,-4/3) -- (0,-1) -- (1,-1) -- (1,1) -- (0,1) -- (-2/3,2/3) -- cycle ;
\draw[densely dotted] (-1,0) -- (-2/3,0) ;
\draw[densely dotted] (0,0) -- (1,0) ;
\draw[densely dotted] (-5/3,-1/3) -- (1/3,-1/3) ;
\draw[densely dotted] (1,1/3) -- (5/3,1/3) ;
\draw[densely dotted] (-5/3,-1/3) -- (-2/3,1/6) ;
\draw[densely dotted] (-2/3,-1/3) -- (0,0) ;
\draw[densely dotted] (1/3,-1/3) -- (5/3,1/3) ;
\draw[thick] (0,-1) -- (0,1) ;
\draw[thick] (1,-1) -- (1,1) ;
\draw[densely dotted] (-1,-1) -- (-1,1) ;
\draw[thick] (-2/3,-4/3) -- (-2/3,2/3) ;
\draw[densely dotted] (2/3,1) -- (2/3,4/3) ;
\draw[densely dotted] (-1,1) -- (1,1) ;
\draw[densely dotted] (-1,-1) -- (-2/3,-1) ;
\draw[densely dotted] (-2/3,2/3) -- (2/3,4/3) ;
\draw[densely dotted] (-2/3,-4/3) -- (0,-1) ;
\draw[thick] (-2/3,-4/3) -- (0,-1) -- (1,-1) ;
\draw[thick] (-2/3,2/3) -- (0,1) -- (1,1) ;
\draw (0,-1.6) node {\scalebox{0.75}{\(\{xz=0\wedge x\geqslant0\wedge z\leqslant0\}\)}} ;
\end{scope}
\begin{scope}[xshift=105mm,yshift=-35mm]
\fill[mygray] (-5/3,-1/3) -- (-2/3,-1/3) -- (2/3,1/3) -- (2/3,4/3) -- (-2/3,2/3) -- (-2/3,1/6) -- cycle ;
\draw[densely dotted] (-1,0) -- (-2/3,0) ;
\draw[densely dotted] (0,0) -- (1,0) ;
\draw[densely dotted] (-5/3,-1/3) -- (1/3,-1/3) ;
\draw[densely dotted] (2/3,1/3) -- (5/3,1/3) ;
\draw[thick] (-2/3,-1/3) -- (2/3,1/3) ;
\draw[densely dotted] (1/3,-1/3) -- (5/3,1/3) ;
\draw[densely dotted] (0,-1) -- (0,1) ;
\draw[densely dotted] (1,-1) -- (1,1) ;
\draw[densely dotted] (-1,-1) -- (-1,-1/3) ;
\draw[densely dotted] (-1,0) -- (-1,1) ;
\draw[densely dotted] (-2/3,-4/3) -- (-2/3,2/3) ;
\draw[densely dotted] (2/3,-2/3) -- (2/3,4/3) ;
\draw[densely dotted] (-1,1) -- (1,1) ;
\draw[densely dotted] (-1,-1) -- (1,-1) ;
\draw[thick] (-2/3,2/3) -- (2/3,4/3) ;
\draw[densely dotted] (-2/3,-4/3) -- (2/3,-2/3) ;
\draw[thick] (-2/3,-1/3) -- (-2/3,2/3) ;
\draw[thick] (-2/3,-1/3) -- (-5/3,-1/3) ;
\draw[thick] (-5/3,-1/3) -- (-2/3,1/6) ;
\draw[thick] (2/3,1/3) -- (2/3,4/3) ;
\draw (0,-1.6) node {\scalebox{0.75}{\(\{xy=0\wedge x\leqslant0\wedge y\geqslant0\}\)}} ;
\end{scope}
\begin{scope}[xshift=140mm,yshift=-35mm]
\fill[mygray] (-2/3,-4/3) -- (2/3,-2/3) -- (2/3,-1/6) -- (5/3,1/3) -- (2/3,1/3) -- (-2/3,-1/3) -- cycle ;
\draw[densely dotted] (-1,0) -- (1,0) ;
\draw[densely dotted] (-5/3,-1/3) -- (1/3,-1/3) ;
\draw[densely dotted] (-1/3,1/3) -- (5/3,1/3) ;
\draw[densely dotted] (-5/3,-1/3) -- (-1/3,1/3) ;
\draw[thick] (-2/3,-1/3) -- (2/3,1/3) ;
\draw[thick] (1/3,-1/3) -- (5/3,1/3) ;
\draw[densely dotted] (0,-1) -- (0,-1/3) ;
\draw[densely dotted] (0,0) -- (0,1) ;
\draw[densely dotted] (1,-1) -- (1,1) ;
\draw[densely dotted] (-1,-1) -- (-1,1) ;
\draw[densely dotted] (-2/3,-4/3) -- (-2/3,2/3) ;
\draw[densely dotted] (2/3,1/3) -- (2/3,4/3) ;
\draw[densely dotted] (-1,1) -- (1,1) ;
\draw[densely dotted] (-1,-1) -- (-2/3,-1) ;
\draw[densely dotted] (0,-1) -- (1,-1) ;
\draw[densely dotted] (-2/3,2/3) -- (2/3,4/3) ;
\draw[thick] (-2/3,-4/3) -- (2/3,-2/3) ;
\draw[thick] (-2/3,-4/3) -- (-2/3,-1/3) -- (1/3,-1/3) ;
\draw[thick] (2/3,1/3) -- (5/3,1/3) ;
\draw[thick] (2/3,-2/3) -- (2/3,-1/6) ;
\draw (0,-1.6) node {\scalebox{0.75}{\(\{xy=0\wedge x\geqslant0\wedge y\leqslant0\}\)}} ;
\end{scope}
\end{tikzpicture}}
\end{center}
\caption{The \(9\) traversals at the origin in \(\{(x,y,z)\in[0,1]^3;xyz=0\}\).}\label{fig:K}
\end{figure}

\paragraph{Related work and applications.} 
The optimal schedules for executing a concurrent program correspond to 
the geodesics on its smooth model (provided that the 
latter is equipped with the metric induced by the uniform norm) 
\cite[\S6.1]{haucourt2024non}. 
One can therefore imagine a virtual machine that runs programs according to some kind of \emph{least action principle}, or a compiler that reorders instructions to 
optimize performance, which is the starting point for the EPIC style of 
architectures: «the compiler should play the key role in designing the 
plan of execution, and the architecture should provide the requisite 
support for it to do so successfully» \cite{schlansker00epic}. 

We also mention that, building on the results of the present paper, an elegant way of defining the combinatorial blowup of a precubical set, via \emph{relational} precubical sets, was introduced in~\cite{chamoun2025realization}.

\paragraph{Plan of the paper.} We generalize to \emph{any} locally ordered space the blowup construction used to prove Theorem~\ref{thm:blowup_graphs}. 
The construction is functorial, and provides the right adjoint to the inclusion of the full subcategory of euclidean local orders into that of locally ordered spaces (\S\ref{sec_blowup}).
Before this, we recall some well-known facts about precubical sets and locally ordered spaces (\S\ref{sec_prelim}), and generalize the gluing of topological spaces to the locally ordered setting (\S\ref{sec_colimit}). 
Then we focus on locally ordered realizations of precubical sets (which form a large class of locally ordered spaces), for which blowups have a purely combinatorial description, and which can be used in practice; this provides a direct generalization of the approach in~\cite{haucourt2024non}. By passing, we prove a bunch of theorems which are interesting on their own, about locally ordered realizations of precubical sets (\S\ref{sec_pset}). We conclude with a discussion about the path lifting property, which is important since directed paths represent the execution traces of a program in the context of modeling concurrent systems, and we do not want to `forget' these traces (\S\ref{sec_path}).

%% file: subdiv.tex
\noindent We give the proofs of \S\ref{subdiv}. For \(t=(t_1,\ldots,t_m)\) we define \(\int t=t_1+\ldots+t_m\). 
\lorderedsubdivisionrestate*
\begin{proof}
A lower \((n-k)\)-face is of the form \(f^-(x)\) with \(f^-\) denoting the composite of a sequence of \(k\) faces of the form \(\delta^-_{i,n'}\) with \(k<n'\leqslant n\). 
The same observation applies to upper \((n-k)\)-faces with \(f^+\) composite of faces of the form \(\delta^+_{i,n'}\). 
Let \(e=(x,t)\) be as in the statement. 
Put \(f^-(x,t)=(x^-,t^-)\) and \(f^+(x,t)=(x^+,t^+)\), and suppose that \((x^-,t^-)=(x^+,t^+)\). 
In particular we have \(t^-=t^+\), therefore \(\int t^-=\int t^+\). 
We denote \(\dim t-\dim t^-\) (which is also \(\dim t-\dim t^+\)) by \(N\). 
We observe that applying an operator of the form \(\delta^-_{i,n}\) reduces the sum of the coordinates by \(\tfrac 1{2s}\) (no matter \(t_j\) is equal to \(\tfrac1{2s}\) or not), while applying an operator of the form \(\delta^+_{i,n}\) either increases the sum of the coordinates by \(\tfrac 1{2s}\) (if \(t_j\not=\tfrac{2s-1}{2s}\)), or reduces it by \(\tfrac{2s-1}{2s}\) (if \(t_j=\tfrac{2s-1}{2s}\)) because the coordinate \(t_j\) is `dropped'. 
So \(N\) is precisely the number of indices \(j\) such that \(t_j=\tfrac{2s-1}{2s}\), and \(k-N\) the number of index \(j\) such that \(t_j\not=\tfrac{2s-1}{2s}\). 
It follows that \(\int t^-=\int t-\tfrac k{2s}\) while \(\int t^+=\int t-\tfrac {N(2s-1)}{2s} + \tfrac{k-N}{2s}\). The equality \(\int t^-=\int t^+\) boils down to \(k=sN\). 
\end{proof}

\condrestatable*
\begin{proof}
We can suppose by subdivision (Corollary~\ref{cor:locally_ordered_subdivision}) that $P$ has the following property: every cube has distinct lower and upper faces. Now if we consider the $\epsilon$-elementary neighborhood $W_\epsilon$ (Definition~\ref{def:wepsilon}) of $v$ for some $\epsilon$, all the cubes of $W_\epsilon$ are adjacent to $c$, in particular if $c$ is a cube in $W_\epsilon$ whose minimum is $v$, then $W_\epsilon$ does not intersect the upper faces of $c$. Since ordered open sets generate the topology, we can find an ordered neighborhood $U\subseteq W_\epsilon$ of $v$. In particular, $U$ has all the properties we just stated for $W_\epsilon$. Now if we consider a cubewise increasing path $\gamma$ in $U$ starting at $x\in \{c\}\times]0,1[^{\dim(c)}\cap U$ with $c$ a cube having $v$ as minimum, if there is $t\in]0,1]$ such that $\gamma(t)\not\in\{c\}\times]0,1[^{\dim(c)}$, then either $\gamma$:
\begin{enumerate}
    \item  went up in dimension, an then it went in a cube which has $c$ on its lower face;
    \item went down in dimension, but then it means that it reached the upper face of the cube which contains it, which is impossible by the above mentioned properties of $U$.
\end{enumerate} 
    \vspace{0.2cm}
This means that $\gamma$ can only go up in dimension, and stays in cubes which have $v$ as minimum. Note that in particular, this is true if $\gamma$ starts at $v$.
\end{proof}

\Wconnectedrestatable*
    \begin{proof}
        The only $n$-cubes of $\frac{1}{s}P$ having $v=(v,())$ as vertex are of the form $(c,(t_1,\dots,t_n))$ with $t_i\in\{\frac 1 2s,\frac {2s-1} 2s\}$ for some cube $c$ of $P$. Now for each such cube, there is at most one $1$-face map which can give $(v,())$, because $1/2s\ne(2s-1)/2s$ since $s>1$.
    \end{proof}

%% file: proof_invariance.tex
First we claim that $P\cap f(V)$ is $n$-homogeneous. So suppose that there is a $k$-cube $c$ of $P\cap f(V)$, $k<n$ such that $f(V)$ does not intersect any of the cubes of $P$ which have $c$ as $k$-face. By Remark~\ref{rem:wepsilon}, we can find $x\in\{c\}\times ]0,1[^k$ and an open neighborhood $T$ of $x$ such that $T$ only intersects cubes of higher dimension adjacent to $c$, i.e. such that $T\cap f(V)\subseteq \{c\}\times ]0,1[^k$. By continuity, $f^{-1}(T)$ is open in $\R^n$, so we get a continuous injection $f$ from an open subset $V\cap f^{-1}(T)$ of $\R^n$ to an open subset of $\R^k$, which is absurd by (classical) invariance of domain.\\
\\
Since $v$ is simple, for every $n$-cube $c$ with vertex $v$, $\{c\}\times ]0,1[^n\cap W$ is connected. Now we prove that for every $n$-cube $c$ of $P\cap f(V)$, we have $\{c\}\times ]0,1[^n\cap W\subseteq f(B)$. Indeed, consider the set $$Z:=\{x\in \{c\}\times ]0,1[^n\cap W\,|\,x\in f(B)\}$$
Notice that $$Z=\{x\in \{c\}\times ]0,1[^n\cap W\,|\,x\in f(\mathring{B})\}$$
because if $x\in W\cap f(B)$, there is $y\in B$ such that $x=f(y)$, so $f(y)\in W$, but $W\cap f(S)=\varnothing$ by definition of $W$, and so $y\in\mathring{B}$. By the choice of $c$, $Z$ is not empty. So it remains to show that it is open and closed. Notice that by injectivity of $f$, $$Z=f(f^{-1}(\{c\}\times ]0,1[^n\cap W)\cap \mathring{B})$$
but $\{c\}\times ]0,1[^n\cap W$ is homeomorphic to an open subset of $\R^n$ and is open in $|P|$, so $Z$ is open by classical invariance of domain. $Z$ is also closed, because $B$ is compact, which concludes.\\
\\
We mentioned that $f(B)$ is closed. Now notice that since $|P\cap f(V)|$ is $n$-homogeneous, the closure of $(P\cap f(V))_n\times ]0,1[^n\cap W$ is equal to $|P\cap f(V)|\cap \Bar{W}$, where $\Bar{W}$ is the closure of $W$. So the result of the last paragraph gives $|P\cap f(V)|\cap W\subseteq f(B)$. In fact, since $W\cap f(S)=\varnothing$, this means that $|P\cap f(V)|\cap W\subseteq f(\mathring{B})$. But then if $x\in |P\cap f(V)|\cap W$, there is $y\in \mathring{B}$ such that $x=f(y)$, and then $f(y)\in W$ so $y\in f^{-1}(W)$, so $y\in V$ and $x\in f(V)$. This finishes the proof.

%% file: proof_local.tex
\begin{definition}
    For $p\in \{<,>,=\}^n$, define the $p$-cone of $\R^n$ as 
    \[\R^n_p:=\{(x_1,\dots,x_n)\in\R^n\,|\,\forall i\,x_ip_i0\}\]
    We will allow ourselves to represent $p$ as an element of $\{-1,1,0\}^n$ in the obvious way.
\end{definition}

\begin{theorem}
    Let $P$ be a precubical set, $v$ a vertex. The following are equivalent:
    \begin{enumerate}
        \item $(P,v)$ is a local precubical structure of $\R^n$;
        \item $v$ has $n$ edges in and $n$ edges out (counted `with multiplicity', i.e.~if the same edge is in and out, it counts as two edges), labeled respectively by $\{-n,-(n-1),\dots,-1\}$ and $\{1,\dots,n\}$, and such that for every $k\geq 2$, there is exactly one $k$-cube $c(p)$ adjacent to $v$ for every tuple $p$ of $k$ elements of $\{-n,-(n-1),\dots,-1\}\cup\{1,\dots,n\}$ not containing $i$ and $-i$ for any $i$, such that the $1$-faces of $c(p)$ adjacent to $v$ are given by the edges of $p$ (also up to multiplicity, i.e.~we cannot have a cube which has $v$ and the same adjacent edges at two different corners), and such that these are no more cubes adjacent to $v$.
        \item There is a precubical set $B\cong (\cdot\xrightarrow{}\cdot\xrightarrow{}\cdot)^{\otimes n}$ the $n$-fold tensor product and a monomorphism $$B\hookrightarrow{} \frac{1}{3}P$$ 
        sending the central vertex to $v$, and whose image is exactly the precubical set generated by $N(v)$, where $v$ is seen as a vertex of $\frac{1}{3}P$.
    \end{enumerate}
\end{theorem}
\begin{proof}
    (3) clearly implies (1) (consider $W_\epsilon(v)$). We start by showing that (2) implies (3). We define a monomorphism $\iota:T:=(-1\xrightarrow{a}0\xrightarrow{b}1)^{\otimes n}\xrightarrow{}\frac{1}{3}P$ as follows: 
    \begin{enumerate}
        \item $\iota(0,\dots,0):=v$;
        \item $\iota(0,\dots,0,a,0,\dots,0):=(-k,(5/6))$ and\\
         $\iota(0,\dots,0,b,0,\dots,0):=(k,(1/6))$) where $a$ (respectively $b$) is in position $k$ (recall that $-k$ (resp. $k$) is the label of an edge with target (resp. source) $v$);
        \item for every tuple $p\in\{a,0,b\}^n$, $\iota(p)=(c(p'),(t_{p_i})_{i\in\{1,\dots n\}})$ where $p'$ is obtained from $p$ by replacing every $a$ (resp. $b$) in position $k$ by $-k$ (resp. $k$) and forgetting the $0$s, and where $t_{a}:=5/6$, $t_b:=1/6$ and $t_0$ is omitted.
    \end{enumerate} 
    \vspace{0.2cm}
    (recall that we are working up to symmetry, so we do not care about the order of the dimensions). This is enough since $T$ is $n$-homogeneous. We still need to check that this is compatible with the face maps, but this is just a formal exercise. It is also easy to check that all the vertices in the image are distinct, which proves that this is indeed a monomorphism, because if $\iota(c)=\iota(c')$ for some $k$-cubes $c,\,c'$, then at least $2^k+1$ vertices are sent to $2^k$ vertices, so some vertices are identified, which contradicts what we just said. Since $(P,v)$ satisfies (2) of Theorem~\ref{explicit_blowup}, then by construction all the neighboring cubes of $v$ in $\frac{1}{3}P$ are in the image of $\iota$.\\
    \\
    We still need to prove (1) implies (2). Fix a dihomeomorphism $f:W\xrightarrow{}\R^n$. In this proof, we work up to multiplicity. First, we show that there is exactly one $n$-cube which has $v$ as minimum (resp. maximum), which we will call the positive (respectively negative) cubes. By conditions (iii) and (iv) of \S\ref{sec:lorps} and the fact that $f$ is order preserving and order reflecting, the positive cone of $\R^n$ (i.e.~$\R^n_+:=\{(x_1,\dots,x_n)\in\R^n\,|\,\forall i,\,x_i\geq 0\}$) is exactly the image by $f$ of all the cubes (of any dimension) which have $v$ as minimum. Obviously, there is at least one such cube of dimension $n$ (by invariance of domain), so we need to prove that there is at most one. Now notice that the positive cone of $\R^n$ has the property that every two elements have an upper bound, which is just an upper bound pointwise. Suppose that there are more than two $n$-cubes which have $v$ as minimum. Choose two such cubes $c$ and $c'$. We take $y\in \{c\}\times]0,1[^n$, $y'\in \{c'\}\times]0,1[^n$, and we claim that they cannot have an upper bound. Indeed, since $n$ is the dimension of $P$ and $U$ satisfies condition (iv) of \S\ref{sec:lorps}, the elements which are bigger than $y$ (respectively $y'$) in $W$ all belong to $\{c\}\times]0,1[^n$ (respectively $\{c'\}\times]0,1[^n$), so $y$ and $y'$ cannot have an upper bound in $W$. But $f$ preserves and reflects upper bounds, so we get a contradiction. The same proof works in the negative case. We write $c_+$ and $c_-$ for the positive and negative cube respectively.\\ 
    \\
    Now we prove that the only edges in $W$ are the edges of the positive and negative cubes. Indeed, suppose that there is an edge which goes out of $v$, which is not one of the edges of $c_+$. Notice that the cubes having $v$ as minimum generate a subprecubical set of $P$ which, by the same arguments as in the proof of Lemma~\ref{invariance_domain}, is $n$-homogeneous. So the existence of this additional edge implies the existence of another positive ($n$-)cube, which is a contradiction. Again the same is true in the negative case. Notice that the same argument shows that every $k$-cube $c$ whose edges in $W$ are all of positive (resp. negative), $c$ is a $k$-face of $c_+$ (respectively $c_-$).\\
    \\
    We have proved that there is a cube $c_+$ that is exactly sent to the positive cone of $\R^n$. Now we claim that the edges of $c_+$ are exactly mapped to the (positive) axis of $\R^n$. This comes from the fact that in both cases, the corresponding points are characterized by the following property 
    \begin{center}
        $F(x, K)$: the set of points of $K$ that are strictly lower than $x$ is dihomeomorphic to an interval
    \end{center}
    Indeed, take a point $x$ on one of the axis of $\R^n$, then all its coordinates are zero except one, say $x=(x_1,0,\dots,0)$, so the only points of the positive cone which are lower than $x$ are those of the form $(x_1',0,\dots,0)$ with $x_1'\leq x_1$. This proves $F(x,\R^n_+)$. Now for $x\in\{e\}\times]0,1[\cap U$ for $e$ a positive edge, any element $y\in|c_+|\cap U$ which is smaller than $x$ belongs to $\{e\}\times]0,1[\cap U$, because otherwise a cubewise increasing path from $y$ to $x$ would necessarily go through some cube of higher dimension, which is impossible by condition (iv) of \S\ref{sec:lorps}. This proves $F(x,|c_+|\cap U)$. Now for an edge $e$ of $c_+$, we need to prove that $|e|\cap W$ is connected: otherwise, there are $a,b\in |e|\cap W$ with no directed path connecting them in $|c_+|\cap W$, because the only one that does is included in $|e|$ (again by condition (iv) of \S\ref{sec:lorps}), but supposing without loss of generality that $a\leq b$, we get $f(a)\leq f(b)$ in $\R^n$, so there is a directed path between them in the positive cone, which can be pulled back to a directed path in $|c_+|\cap W$, which is a contradiction. This proves $F(x,|c_+|\cap W)$. So both the points of the edges of $c_+$ and the points of the positive axis of $\R^n$ satisfy the property, and they are clearly the only ones that do. A dihomeomorphism respects this property (more precisely $F(x,K)\leftrightarrow F(f(x),f(K))$ if $f$ is a dihomeomrphism on $K$), so the claim is proved. Notice that by continuity, every edge of $c_+$ is mapped to exactly one axis, and by bijectivity all the axes are in the image, so all the edges of $c_+$ are distinct. As a consequence, each face is mapped to the corresponding face of $\R^n$: for example in dimension 3, if $c$ is a positive $2$-cube of edges $e_1$ and $e_2$, the third one being $e_3$ (and supposing that $f$ sends $e_i$ to the $i$th axis of $\R^3$), a point $(c,(x,y))$ of the positive cone will be mapped to $(x',y',z'):=f((c,(x,y)))$ in the positive cone of $\R^3$. Now $f((c,(x,y)))$ is the least upper bound of $(x',0,0)$, $(0,y',0)$ and $(0,0,z')$, so $(c,(x,y))$ is the least upper bound of their preimages, which are on the axes $e_i$. But then we get that these are respectively $(e_1,x)$, $(e_2,y)$ and $0$, and since $f$ is injective this means that $z'=0$, so indeed $f(x,y,0)$ lands in the correct face. This argument works in the general case.\\
    \\
    Now we label the edges following the axis to which they are mapped. We still need to prove the part about the $k$-cubes in definition~\ref{canonical}. For better readability, we give negative coefficients to negative edges in the realization (i.e.~we replace $\{c\}\times]0,1[^{k}$ with $k=\dim(c)$ by $\{c\}\times]0,1[^{k_1}\times]-1,0[^{k_2}$ (up to permutation), where $k_1$ and $k_2$ are respectively the number of positive and negative edges of the cube $c$), only in this paragraph. First we show that for a tuple of edges containing $i$ and $-i$ for some $i$, i.e.~a set of edges such that one is sent to a positive axis and another one to the corresponding negative axis, there cannot be a cube having these edges as $1$-faces. It suffices to prove that there cannot be a $2$-face between edges $i$ and $-i$ for some $i$. Suppose the contrary: for some label $i$, there is a $2$-face $c$ between $i$ and $-i$. Consider two symmetrical points $(i,x)$ and $(-i,-x)$ on the edges $i$ and $-i$ respectively (recall that $i$ is out and $-i$ is in). The segment $$\gamma:t\in]0,x[\mapsto (c,(x-t,-t))$$ is directed from $(i,x)$ to $(-i,-x)$, so it corresponds to a directed segment in $\R^n$. However, the only directed segment, up to parametrization, between two points on the same axis in $\R^n$ is contained in this axis, but $\gamma$ is not contained in the preimage of the axes, so we got our contradiction. \\
    \\
    We still need to show that given a tuple of $k$ labels not containing $i$ and $-i$ for some $i$, there is exactly one $k$-cube having these as $1$-faces. Take $(i_1,\dots,i_k)$ tuple of labels, $c$ a $k$-cubes having exactly these as $1$-faces. Consider a point $(c,(t_1,\dots,t_k))$. Suppose for simplicity that $i_1,\dots,i_p$ are positive edges and $i_{p+1},\dots,i_k$ are negative edges. As remarked earlier, there is only one cube $c_1$ (respectively $c_2$) having $i_1,\dots,i_p$ (respectively $i_{p+1},\dots,i_k$) as edges. Let $x_+$ and $x_-$ be the 'projections' of $x$ on these cubes, i.e.~$x_+=(c_1,(t_1,\dots,t_p))$ and $x_-=(c_2,(t_{p+1},\dots,t_k))$. Since $U$ is order convex, we have $x_-< x< x_+$. Now $x_-$ and $x_+$ are sent respectively to a point $(0,\dots,0,y_{p+1},\dots,y_n)$, $y_{p+1},\dots,y_n<0$ and $(y_1,\dots,y_p,0,\dots,0)$, $y_1,\dots,y_p>0$ of $\R^k$, seen as a subset of $\R^n$. Since $f$ is order preserving, $$f(x)\in\{(z_1,\dots,z_k)\in\R^k\,|\,z_1,\dots,z_p> 0,\,z_{p+1},\dots,z_k< 0\}$$ (note that it does belong to $\R^k$ seen as a subset of $\R^n$, because for any other coordinate $j$, $0=f(x_-)_j\leq (f(x))_j\leq f(x_+)_j=0$, so $(f(x))_j=0$ as well). This shows that all the cubes of $P$, of different dimensions and for different choices of signs of the coordinates, are exactly sent to the corresponding cone (i.e.~'with the same signs') in $\R^n$. So for every such choice, there is at least one cube filling the edges, else $f$ would not be surjective. Moreover, suppose that there exists $(i_1,\dots,i_k)$ tuple of labels, $c_1$ and $c_2$ distinct $k$-cubes having exactly these as $1$-faces. Take $x_1\in\{c_1\}\times]0,1[^k$, $x_2\in\{c_2\}\times]0,1[^k$. $f(x_1)$ and $f(x_2)$ have an upper bound $z$ in the interior of the corresponding cone of $\R^n$, with two directed paths $\gamma_1:f(x_1)\xrightarrow{}z$ and $\gamma_2:f(x_2)\xrightarrow{}z$ entirely contained in the interior of this cone, which is the preimage of $C\times]0,1[^k$ with $C$ the set of cubes having $(i_1,\dots,i_k)$ as edges. But $z=f(y)$ for some $y$, and $y$ necessarily belongs to $\{c_3\}\times]0,1[^k$ for some $c_3$ having the same edges, but there is no directed path joining for example $c_1$ and $c_3$ and staying in $C\times]0,1[^k$, contradiction.
\end{proof}

%% file: proof_comb.tex
\begin{lemma}
    Let $P$ be a precubical set of dimension $n$, $c$ a $k$-cube, $0\leq i < k$, $c':=\delta^\epsilon_{k,i}(c)$ the $(k-1)$-face of $c$ in direction $i$. Given $(Q,m(c'))$ a local precubical structure of $\R^n$ in $\frac{1}{2}P$, seen as subobject $B$ of $\frac{1}{6}P$ isomorphic to $(\cdot\xrightarrow{}\cdot\xrightarrow{}\cdot)^{\otimes n}$, let 
    $$S:=\{c'\in \bigsqcup_m Q_m\subseteq \bigsqcup_m\frac{1}{6}P_m\,|\,\pi_1(c')=c\}$$ 
    be the set of cubes of $Q$ whose underlying cube in $P$ is $c$. If $S$ contains the edge $e$ `going inside $c$ in direction $i$', namely $$e=(c,(1/2,\dots,1/2,z,1/2,\dots,1/2))$$ where $z$ is in position $i$ and $z=1/12$ (resp. $11/12$) if $\epsilon=-$ (resp. $+$), then the precubical set $\psi_{P,c,i,\epsilon}(Q,m(c'))$ obtained by `shifting $S$ to the center of $c$', namely the precubical set spanned by 
    \begin{align*}
    \{ & (c, (t_1,\dots, t_{i-1},t_i-\epsilon.1/6,t_{i+1}\dots,t_n)),\\
       & (c, (t_1,\dots, t_{i-1},t_i-\epsilon.1/2,t_{i+1}\dots,t_n))\,|\\
       & (c, (t_1,\dots,t_n))\text{ is an }n\text{-cube of }S\}
    \end{align*}
    is a local precubical structure of $\R^n$. Else $S$ is empty.
\end{lemma}
    \begin{proof}
        $S_i$ contains $e$ if and only if the realization of the cubes of $Q$ adjacent to $m(c')$ intersects $\{c\}\times]0,1[^{\dim(c)}$. In this case, for any point $x\in \{c\}\times]0,1[^{\dim(c)}$, there exists an embedding of $\R^n$ centered at $x$. This is in particular true for $m(c)$. Then by Lemma~\ref{invariance_domain} (and Lemma~\ref{W_connected}) and Theorem~\ref{explicit_blowup}, there is a precubical subset $Q'$ such that $(Q',m(c))$ is a local precubical structure of $\R^n$. But the $n$-cubes of $Q'$ necessarily have the same underlying $n$-cubes in $P$ as $S_i$, and $\psi_{P,c,i,\epsilon}(Q,m(c'))$ is the only possible such. 
    \end{proof}

    Note that we can also give a proof using explicit calculations on the subdivisions. Now the transition morphisms of $\mathsf{Comb}_P$ are defined in the following way: for a morphism $d^\epsilon_{k,i}:(k-1,c')\to (k,c)$ of $\int P$ (with $c':=\delta^\epsilon_{k,i}(c)$), given $x\in \mathsf{Comb}_P(k-1,c')$, if $\psi_{P,c,i,\epsilon}(x)$ is defined, set $\mathsf{Comb}_P(d^\epsilon_{k,i})(x):=(\psi_{P,c,i,\epsilon}(x),m(c))$. Else set $\mathsf{Comb}_P(d^\epsilon_{k,i})(x):=\bot$. This is well-defined: one just needs to check that the restriction maps respect the cubical relations, which is formal.